\documentclass[reqno,10pt,centertags]{amsart}
\usepackage[letterpaper,margin=1.5in,bottom=1.3in]{geometry}
\usepackage{amsmath,amsthm,amscd,amssymb,latexsym,enumerate}
\usepackage{comment}

\usepackage{hyperref}

\newcommand*{\mailto}[1]{\href{mailto:#1}{\nolinkurl{#1}}}
\newcommand{\arxiv}[1]{\href{http://arxiv.org/abs/#1}{arXiv: #1}}

\makeatletter
\def\theequation{\@arabic\c@equation}

\newcommand{\bbN}{{\mathbb{N}}}
\newcommand{\bbR}{{\mathbb{R}}}

\newcommand{\bbS}{{\mathbb{S}}}

\renewcommand{\a}{\alpha}
\renewcommand{\b}{\beta}
\newcommand{\g}{\gamma}

\renewcommand{\t}{\theta}

\newcommand{\norm}[1]{\left\Vert#1\right\Vert}

\newcommand{\no}{\nonumber}
\newcommand{\lb}{\label}
\newcommand{\bi}{\bibitem}
\newcommand{\f}{\frac}

\newcommand{\ol}{\overline}
\newcommand{\bs}{\backslash}

\newcommand{\wti}{\widetilde}

\newcommand{\dott}{\,\cdot\,}

\newcommand{\pr}{\frac{\partial}{\partial r}}

\renewcommand{\Re}{\operatorname{Re}}
\renewcommand{\Im}{\operatorname{Im}}
\renewcommand{\ln}{\operatorname{ln}}

\allowdisplaybreaks
\numberwithin{equation}{section}

\newtheorem{theorem}{Theorem}[section]
\newtheorem{lemma}[theorem]{Lemma}
\newtheorem{corollary}[theorem]{Corollary}

\theoremstyle{definition}

\newtheorem{remark}[theorem]{Remark}

\begin{document}

\title[Weighted Rellich-Type Inequalities]{Factorizations and Power Weighted Rellich and Hardy--Rellich-Type Inequalities}

\author[F.\ Gesztesy]{Fritz Gesztesy}
\address{Department of Mathematics, 
Baylor University, Sid Richardson Bldg., 1410 S.~4th Street, Waco, TX 76706, USA}
\email{\mailto{Fritz\_Gesztesy@baylor.edu}}
\urladdr{\url{http://www.baylor.edu/math/index.php?id=935340}}

\author[M.\ M.\ H.\ Pang]{Michael M.\ H.\ Pang}
\address{Department of Mathematics,
University of Missouri, Columbia, MO 65211, USA}
\email{\mailto{pangm@missouri.edu}}
\urladdr{\url{https://www.math.missouri.edu/people/pang}}

\author[J.\ Parmentier]{Jake Parmentier}
\address{Department of Mathematics,
University of Missouri, Columbia, MO 65211, USA}
\email{\mailto{jmpft4@missouri.edu}}

\author[J.\ Stanfill]{Jonathan Stanfill}
\address{Department of Mathematics, The Ohio State University \\
100 Math Tower, 231 West 18th Avenue, Columbus, OH 43210, USA}
\email{\mailto{stanfill.13@osu.edu}}
\urladdr{\url{https://u.osu.edu/stanfill-13/}}

\dedicatory{Dedicated, with admiration, to Brian Davies on the occasion of his 80th birthday} 
\date{\today}
\@namedef{subjclassname@2020}{\textup{2020} Mathematics Subject Classification}
\subjclass[2020]{Primary: 35A23, 35J30; Secondary: 47A63, 47F05.}
\keywords{Weighted Rellich-type inequality, factorizations of differential operators.}

\begin{abstract}
We revisit and extend a variety of inequalities related to power weighted Rellich and Hardy--Rellich inequalities, including an inequality due to Schmincke. 
\end{abstract}

\maketitle


\section{Introduction} \lb{s1}

Brian Davies had a keen interest in Hardy and Rellich-type inequalities (see, e.g., \cite[Sect.~1.5]{Da89}, \cite[Sect.~5.3]{Da95}, \cite{Da95a}, \cite{Da99}, \cite{DH98}), and we hope that our modest contribution to this subject will cause him some joy. 

To set the stage, we briefly indicate the kind of inequalities that are considered in this paper:  
One of the inequalities we recover is the following sharp inequality by Caldiroli and Musina 
\cite[Theorem~3.1]{CM12}
\begin{align}\lb{1.1}
\begin{split}
\int_{\bbR^n} |x|^\g|(\Delta f)(x)|^2 \, d^n x 
 \geq C_{n,\g} \int_{\bbR^n} |x|^{\g-4} |f(x)|^2 \, d^n x,& \\
 \g\in\bbR, \; f \in C^{\infty}_0(\bbR^n \backslash \{0\}), \; n\in\bbN, \, n\geq2,&
 \end{split} 
\end{align}
where
\begin{equation}\lb{1.2}
C_{n,\g}=\min_{j\in\bbN_0}\Big\{\big[4^{-1}(n-2)^2 - 4^{-1}(\g-2)^2+j(j+n-2)\big]^2\Big\}.
\end{equation}
In addition, we will also derive the sharp inequality (sometimes called the Hardy--Rellich inequality), 
\begin{align}\lb{1.4}
\begin{split} 
\int_{\bbR^n} |x|^\gamma|(\Delta f)(x)|^2\, d^nx\geq A_{n,\gamma}\int_{\bbR^n}|x|^{\gamma-2}|(\nabla f)(x)|^2\, d^nx,&  \\
\gamma \in \bbR, \; f\in C_0^\infty(\bbR^n\backslash \{0\}),\; n\in\bbN,\, n\geq 2,&
\end{split} 
\end{align}
where
\begin{equation}\label{1.5}
A_{n,\gamma}={\min}_{j \in \bbN_0}\{\alpha_{n,\gamma,\lambda_j}\},
\end{equation}
with
\begin{align}
\alpha_{n,\gamma,\lambda_0}&=\alpha_{n,\gamma,0}=4^{-1}(n-\gamma)^2,     \no \\
\alpha_{n,\gamma,\lambda_j}&=\big[4^{-1}(n+\gamma-4)(n-\gamma)+j(j+n-2)\big]^2\Big/\big[4^{-1}(n+\gamma-4)^2+j(j+n-2)\big],     \no \\
& \hspace*{10.5cm} j\in\bbN.      \lb{1.5b}
\end{align}

In the unweighted case $\gamma = 0$ this simplifies to the known fact,
\begin{equation}
A_{n,0} = \begin{cases}  n^2/4, & n \geq 5, \\
3, & n=4, \\
25/36, & n=3, \\
0, & n=2,
\end{cases}      \lb{1.6}
\end{equation}
in particular, there is no nontrivial inequality of the type \eqref{1.4} in the case $n=2$, $\gamma = 0$.

Actually, we derive a number of additional Hardy--Rellich-type inequalities, and, as a representative example, we mention the following weighted version of Schmincke's inequality \cite{Sc72}: Consider the inequality
\begin{align}
\int_{\bbR^n} |x|^\g|(\Delta f)(x)|^2 \, d^n x 
& \geq - s \int_{\bbR^n} |x|^{\g-2} |(\nabla f)(x)|^2 \, d^n x  \no \\
& \quad + [(\g+n - 4)/4]^2 \big[(\g-n)^2+4s\big] \int_{\bbR^n} |x|^{\g-4} |f(x)|^2 \, d^n x.   \lb{1.8} \\
& \hspace*{2.8cm} \gamma \in\bbR, \; f\in C_0^\infty(\bbR^n\backslash\{0\}), \; n\in\bbN, \, n\geq2.   \no 
\end{align}
Then, the following assertions $(i)$ and $(ii)$ hold:
\begin{align} 
\begin{split} 
& \text{$(i)$ If } \, \gamma \in \big[2-(n-1)^{1/2}, 2+(n-1)^{1/2}\big], \, \text{ the estimate \eqref{1.8} holds} \\
& \quad \text{for $s \in \big[-2^{-1}\big\{(n-2)^2+(\g-2)^2\big\}, \infty\big)$.}    \lb{1.9}  
\end{split} \\
\begin{split} 
& \text{$(ii)$ If } \, \gamma \in \bbR \big\backslash \big[2-(n-1)^{1/2}, 2+(n-1)^{1/2}\big], \, \text{ the estimate  \eqref{1.8} holds} \\
& \quad \text{for $s \in \big[-2^{-1}\big\{(n-2)^2-(\g-2)^2\big\}+1-n, \infty\big)$.}     \lb{1.10}
\end{split} 
\end{align} 

Given the enormity of the literature on (power weighted) Hardy and Rellich-type inequalities, one is hard pressed to give an appropriate historical account. This applies, in particular, to the case of Hardy-type inequalities and hence we only refer to the standard monographs such as, \cite{ABG87}, \cite[Chs.~1--5]{BEL15}, \cite[p.~350--352]{EE18}, 
\cite[Parts~1--3]{GM13}, \cite[Ch.~6]{KLV21}, \cite[Chs.~3--10]{KMP07}, \cite[Chs.~1--4, Sect.~7.7]{KPS17}, \cite[Sects.~1.3, 2.7]{Ma11}, \cite[Chs.~1--3]{OK90}, and \cite[Ch.~2]{RS19}. Also in the case of Rellich and Hardy--Rellich inequalities we feel we have to restrict ourselves to a few selected references and some monographs in this context, such as, \cite{Al76}, \cite{Av17}, \cite{Av18}, \cite{Av21}, \cite{ANS11}, \cite[Chs.~1, 3, 6]{BEL15}, \cite{Ca20}, \cite{Ca21}, \cite{CM12}, \cite{CW01}, \cite{Co09}, \cite{DH98}, \cite{DLT19}, \cite{DLT20}, \cite{EE16}, \cite{EEL21}, \cite{Ev09}, \cite{GL18}, \cite{GM11}, \cite[Part~2]{GM13}, \cite{Hi04}, \cite[Ch.~6]{KLV21}, \cite{KY17}, \cite{MOW17}, \cite[Sects.~1.3, 2.7]{Ma11}, \cite{MNSS21}, \cite{MSS15},  \cite{Mu14}, \cite{Mu14b}, \cite[Ch.~2, Sect.~21]{OK90}, \cite{Ow99}, \cite{Re56}, \cite[Sect.~II.7]{Re69}, \cite{RS17}, \cite[Sects.~2.1.5, 2.3.1, 3.1, 3.3]{RS19}, \cite{RY20}, \cite{Si83}, \cite{Ya99}, and the references cited therein.

Thus far we basically focused on inequalities in $L^2(\bbR^n)$, yet much of the recent work on Rellich and higher-order Hardy inequalities aims at $L^p(\Omega)$ for open sets $\Omega \subset \bbR^n$ (frequently, $\Omega$ is bounded with $0 \in \Omega$), $p \in [1,\infty)$, appropriate remainder terms (the latter often associated with logarithmic refinements or with boundary terms), higher-order Hardy--Rellich inequalities, and the inclusion of magnetic fields and weights. Once more, it is impossible to achieve any reasonable level of completeness of the list of references cited in this paper. Hence we again restrict ourselves to Rellich and higher-order Hardy inequality references only and thus refer, for instance, to \cite{AGS06}, \cite{AS09}, \cite{Al76}, \cite{Av16}, \cite{Av16a}, \cite{Av19}, \cite[Ch.~6]{BEL15}, \cite{Ba06}, \cite{Ba07}, \cite{BT06}, \cite{Be09}, \cite{BCG10}, \cite{Ca21}, \cite{DH98}, \cite{DHA04}, \cite{DHA04a}, \cite{DHA12}, \cite{DFP14}, \cite{DLL22}, \cite{DLP21}, \cite{EE16}, \cite{Ev09}, \cite{EL05}, \cite{EL07}, \cite{Ga06}, \cite{GGM03}, \cite{GM11}, \cite[Part~2]{GM13},\cite{Ka84}, \cite{La18}, \cite{MNSS21}, \cite{MSS15}, \cite{MSS16}, \cite{Mo12}, \cite{Mu14}, \cite{Na22}, \cite{NN20}, \cite{Ow99}, \cite{Pa91}, \cite[Sect.~3.1]{RS19}, \cite{Sa22}, \cite{TZ07}, \cite{Xi14}, \cite{XY09}, and the extensive literature cited therein. 

More specifically, we mention that Caldiroli and Musina \cite{CM12} proved in 2012 that the constant $C_{n,\gamma}$ in \eqref{1.1} is optimal. (For various restricted ranges of $\gamma$ see also Adimurthi, Grossi, and Santra \cite{AGS06}, Ghoussoub and Moradifam \cite{GM11}, \cite[Sects.~6.3, 6.5, Ch.~7]{GM13}, and Tertikas and Zographopoulos \cite{TZ07}.) The special unweighted case $\gamma =0$ was settled for $n \geq 5$ by Herbst \cite{He77} in 1977 and subsequently by Yafaev \cite{Ya99} in 1999 for $n \geq 3$, $n \neq 4$ (both authors consider much more general fractional inequalities).

Under various restrictions on $\gamma$, Tertikas and Zographopoulos \cite{TZ07} obtained in 2007 optimality of $A_{n,\gamma}$ for $n \geq 5$ and $\bbR^n$ replaced by appropriate open bounded domains $\Omega$ with $0 \in \Omega$. This is revisited in Ghoussoub and Moradifam \cite{GM11}, \cite[Part~2]{GM13}. 
Similarly, Tertikas and Zographopoulos \cite{TZ07} obtained optimality of $A_{n,0}$ for $n \geq 5$; Beckner \cite{Be08a} (see also \cite{Be08}), and subsequently, Ghoussoub and Moradifam \cite{GM11}, \cite[Sects.~6.3, 6.5, Ch.~7]{GM13} and Cazacu \cite{Ca20}, obtained optimality of $A_{n,0}$ for $n \geq 3$. 

Before briefly describing the content of each section, we next describe the essence of the factorization method with the help of the simplest possible illustration in the power weighted Hardy inequality context (for more details see Remark \ref{r2.9}): Given $\alpha,\g \in \bbR$, $n \in \bbN$, $n \geq 2$, one introduces the two-parameter family of homogeneous vector-valued differential expressions 
\begin{equation}
T_{\alpha,\g} := |x|^{\g/2}\big[\nabla + \alpha |x|^{-2} x\big], \quad  x \in \bbR^n \backslash \{0\},
\end{equation}
with formal adjoint, denoted by $T_{\alpha,\g}^+$, 
\begin{equation}
T_{\alpha,\g}^+ = |x|^{\g/2}\big[- {\rm div(\, \cdot \,)} + \big(\alpha-2^{-1}\g\big) |x|^{-2} x \, \cdot\big], \quad  
\; x \in \bbR^n \backslash \{0\},    
\end{equation}
such that 
\begin{equation}
T_{\alpha,\g}^+ T_{\alpha,\g} = - |x|^\g \Delta - \g|x|^{\g-2}x\cdot\nabla + \alpha (\alpha + 2 - n-\g) |x|^{\g-2}. 
\end{equation}
Thus, for $f \in C_0^{\infty}(\bbR^n \backslash \{0\})$, 
\begin{align}
\begin{split}
0  \leq &\int_{\bbR^n} |(T_{\alpha,\g} f)(x)|^2 \, d^n x 
= \int_{\bbR^n} \ol{f(x)} (T_{\alpha,\g}^+ T_{\alpha,\g} f)(x) \, d^n x     \\
=& -\int_{\bbR^n} |x|^\g\ol{f(x)}(\Delta f)(x) \, d^n x-\g\int_{\bbR^n} |x|^{\g-2}\ol{f(x)}[x\cdot(\nabla f)(x)] \, d^n x\\
&+ \alpha (\alpha + 2 - n-\g)
\int_{\bbR^n} |x|^{\g-2} |f(x)|^2 \, d^n x.  \lb{1.12} 
\end{split} 
\end{align}
Standard integration by parts, substituting into \eqref{1.12}, rearranging terms, and maximizing with respect to $\alpha$ yields the sharp weighted Hardy inequality,
\begin{align}
\begin{split}
\int_{\bbR^n} |x|^\g|(\nabla f)(x)|^2 \, d^n x \geq [(n - 2+\g)/2]^2 
\int_{\bbR^n} |x|^{\g-2} |f(x)|^2 \, d^n x,& \\
f \in C_0^{\infty}(\bbR^n \backslash \{0\}), \; n \geq 2, \; \gamma \in\bbR \backslash \{2-n\}.&      \lb{1.14} 
\end{split}
\end{align}

We will employ the idea of factorizations in terms of homogeneous vector-valued differential expressions also in the more general case of Rellich-type inequalities, but with an additional twist: In order to arrive at sharp inequalities it will be necessary to involve an operator-valued coefficient involving the Laplace--Beltrami operator 
$-\Delta_{\bbS^{n-1}}$ in $L^2(\bbS^{n-1}; d^{n-1} \omega)$, see, for instance, \eqref{3.7}. While essentially all approaches to Rellich and Hardy--Rellich inequalities rely on decompositions into spherical harmonics, what permits us to go beyond existing approaches is our one-dimensional Lemma \ref{l3.13a} on power weighted Hardy and Rellich inequalities (a simplified version of \cite[Lemma~2.1]{GMP22}) that will be applied to the radial variable $r \in (0,\infty)$ in Appendix \ref{sA} with minimal assumptions regarding $\gamma \in \bbR$. 

In Section \ref{s2} we revisit factorizations without the additional spherical term containing $-\Delta_{\bbS^{n-1}}$. Following \cite{GL18}, we thus derive a first set of inequalities that contains power weighted Rellich, Hardy--Rellich, and Schmincke-type inequalities as special cases. Our novel factorizations including the additional spherical term are then the subject of Section \ref{s3} and once again we derive an improved set of inequalities that contains power weighted Rellich and Schmincke-type inequalities as special cases. A sharp power weighted Hardy--Rellich inequality is then derived in Appendix \ref{sA}.  

Of course, by restriction, the principal inequalities in this paper (such as \eqref{1.1}--\eqref{1.5b}, 
\eqref{1.8}--\eqref{1.10}) extend to the case where $f \in C_0^{\infty} (\bbR^n \backslash \{0\})$, $n \in \bbN$, $n \geq 2$, is replaced by $f \in C_0^{\infty} (\Omega \backslash \{0\})$, where $\Omega \subseteq \bbR^n$ is open with $0 \in \Omega$, without changing the constants in these inequalities. 

As a notational comment we remark that we abbreviate $\bbN_0 = \bbN \cup \{0\}$, and denote by ${\bbS^{n-1}}$ the  unit sphere in $\bbR^n$, and by $B_n(r_0;R)$ the open ball in $\bbR^n$ centered at $r_0 \in (0,\infty)$ of radius $R \in (0,\infty)$, $n \in \bbN$, $n \geq 2$. 

In the remainder of this paper it will repeatedly be convenient to sometimes restrict to real-valued functions and in this context we note the following elementary observation:

\begin{remark} \lb{r1.1}
For $f = f_1 + i f_2 \in C_0^2(\bbR^n)$, with $f_1 = \Re(f)$, $f_2 = \Im(f)$, one has the obvious equalities
\begin{align}
\begin{split}
& |f|^2 = |f_1|^2 + |f_2|^2,    \\
& |\nabla f|^2 =  |\nabla f_1|^2 +  |\nabla f_2|^2,    \lb{1.X} \\
& |\Delta f|^2 = |\Delta f_1|^2 + |\Delta f_2|^2,    
\end{split}
\end{align}
and thus the inequalities stated in Theorems \ref{t2.1}, \ref{t2.7}, \ref{t3.11}, \ref{t3.17a},  \ref{t3.19}, \ref{t3.22}, Corollaries \ref{c2.3}, \ref{c2.5}, \ref{c3.18}, and Lemma \ref{l3.21} hold for all 
$f \in C^{\infty}_0(\mathbb{R}^n \backslash \{0\})$ if and only if they hold for all real-valued 
$f \in C^{\infty}_0(\mathbb{R}^n \backslash \{0\})$. Thus, without loss of generality (w.l.o.g.), we will frequently assume that the functions $f \in C^{\infty}_0(\mathbb{R}^n \backslash \{0\})$ in the proofs of these results are real-valued. However, some of the proofs, for example those of Lemma \ref{l3.5} and Lemma \ref{l3.6}, only work for real-valued functions $f$ and this will then be clearly indicated. 
\hfill $\diamond$
\end{remark} 

\section{Factorizations and Hardy--Rellich-type Inequalities} \lb{s2}

Our main result for this section is the following theorem which is an extension of \cite[Thm. 2.1]{GL18}, in particular, inequalities (2.1) and (2.2) found therein. We will assume that 
$f\in C_0^\infty(\bbR^n\backslash\{0\})$ (infinitely differentiable continuous functions with compact support and support away from zero) throughout. 

Though many of the results proved in this section hold for $n = 1$, we will focus on the multidimensional case, that is, we will assume $n \geq 2$. For results in one dimension, we refer the reader to the references cited in \cite{GLMP22}. 

\begin{theorem} \lb{t2.1}
Let $\a,\b,\g\in\bbR$ and $f\in C_0^\infty(\bbR^n\backslash\{0\})$, $n\in\bbN$, $n\geq2$. Then,
\begin{align}\lb{2.1}
\int_{\bbR^n}|x|^\g |(\Delta f)(x)|^2\, d^nx&\geq [\a(\g+n-4)-2\b]\int_{\bbR^n}|x|^{\g-2}|(\nabla f)(x)|^2\, d^nx \no \\
&\quad-\a(\a-4+2\g)\int_{\bbR^n}|x|^{\g-4}|x\cdot(\nabla f)(x)|^2\, d^nx\\
&\quad+\b[(n-4)(\a-2)-\b+\g(n+\a+\g-6)]\int_{\bbR^n}|x|^{\g-4}|f(x)|^2 d^nx.\no 
\end{align}

In addition, if $\a(\a-4+2\g)\geq0$, then,
\begin{align}\lb{2.2}
\int_{\bbR^n}|x|^\g|(\Delta f)(x)|^2\, d^nx & \geq [\a(n-\a-\g)-2\b]\int_{\bbR^n}|x|^{\g-2}|(\nabla f)(x)|^2\, d^nx\\
&\quad+\b[(n-4)(\a-2)-\b+\g(n+\a+\g-6)]\int_{\bbR^n}|x|^{\g-4}|f(x)|^2 d^nx.\no 
\end{align}
\end{theorem}
\begin{proof}
Given $\a,\b,\g\in\bbR$ and $n\in\bbN$, $n\geq2$, we introduce the three-parameter $n$-dimensional differential expression
\begin{align}
T_{\a,\b,\g}=|x|^{\g/2}\big[-\Delta+\a |x|^{-2}x\cdot\nabla+\b|x|^{-2}\big],\quad x\in\bbR^n\backslash\{0\},
\end{align}
and its formal adjoint, denoted by $T^+_{\a,\b,\g}$,
\begin{align}
\begin{split}
T^+_{\a,\b,\g}=|x|^{\g/2}\big\{&-\Delta-(\a+\g) |x|^{-2}x\cdot\nabla\\
&+\big[\b-\a\big(n-2\big)-2^{-1}\big(2^{-1}\g+\a+n-2\big)\g\big]|x|^{-2}\big\},\quad x\in\bbR^n\backslash\{0\}.
\end{split}
\end{align}
Assuming $f\in C_0^\infty(\bbR^n\backslash\{0\})$ throughout the proof one computes for $T^+_{\a,\b,\g}T_{\a,\b,\g}$,
\begin{align}
\big(T^+_{\a,\b,\g}T_{\a,\b,\g}f\big)(x)&=|x|^\g\big(\Delta^2 f\big)(x)+\a(4-\a-2\g)|x|^{\g-4}\sum_{j,k=1}^n x_j x_k f_{x_j,x_k}(x)\no \\
&\quad+[\a(\g+n-4)-2\b+\g(\g+n-2)]|x|^{\g-2}(\Delta f)(x) \no \\
&\quad+\big[-(n-3)\a^2+2(n-2)\a+4\b\no \\
&\quad\quad+(4\a-\a\g-\a^2-n\a-2\b)\g\big]|x|^{\g-4}[x\cdot(\nabla f)(x)] \no \\
&\quad+\big[\b^2+(n-4)(2\b-\a\b)-(n+\a+\g-6)\b\g\big]|x|^{\g-4}f(x)\no \\
&\quad+2\g|x|^{\g-2}x\cdot(\nabla(\Delta f))(x).\lb{2.5}
\end{align}
One notes that for $\g=0$ the last term of \eqref{2.5} vanishes.

Choosing (w.l.o.g., see Remark \ref{r1.1}) $f\in C_0^\infty(\bbR^n\backslash\{0\})$ to be real-valued and integrating by parts (observing the support properties of $f$, resulting in vanishing surface terms) implies
\begin{align}
0&\leq \int_{\bbR^n}[(T_{\a,\b,\g}f)(x)]^2\, d^nx=\int_{\bbR^n}f(x)\big(T^+_{\a,\b,\g}T_{\a,\b,\g}f\big)(x)\, d^nx\no \\
&=\int_{\bbR^n}|x|^\g f(x)\big(\Delta^2 f\big)(x)\, d^nx+\a(4-\a-2\g)\int_{\bbR^n}|x|^{\g-4}f(x)\sum_{j,k=1}^n x_j x_k f_{x_j,x_k}(x)\, d^nx\no \\
&\quad+[\a(\g+n-4)-2\b+\g(\g+n-2)]\int_{\bbR^n}|x|^{\g-2}f(x)(\Delta f)(x)\, d^nx \no \\
&\quad+\big[-(n-3)\a^2+2(n-2)\a+4\b\no \\
&\quad\quad+(4\a-\a\g-\a^2-n\a-2\b)\g\big]\int_{\bbR^n}|x|^{\g-4}f(x)[x\cdot(\nabla f)(x)]\, d^nx \no \\
&\quad+\big[\b^2+(n-4)(2\b-\a\b)-(n+\a+\g-6)\b\g\big]\int_{\bbR^n}|x|^{\g-4}|f(x)|^2\, d^nx\no \\
&\quad+2\g\int_{\bbR^n}|x|^{\g-2}f(x)[x\cdot(\nabla(\Delta f))(x)]\, d^nx.\lb{2.6}
\end{align}

We now make a few observations in order to simplify \eqref{2.6}. First, standard integration by parts yields
\begin{align}
\begin{split}
\int_{\bbR^n}|x|^\g f(x)\big(\Delta^2 f\big)(x)\, d^nx&=\int_{\bbR^n}|x|^\g |(\Delta f)(x)|^2\, d^nx\\
&\quad-\g(\g+n-2)\int_{\bbR^n}|x|^{\g-2} f(x)(\Delta f)(x)\, d^nx\\
&\quad-2\g\int_{\bbR^n}|x|^{\g-2} f(x)[x\cdot(\nabla(\Delta f))(x)]\, d^nx,
\end{split}
\end{align}
and
\begin{align}
\begin{split}
\int_{\bbR^n}|x|^{\g-2}f(x)(\Delta f)(x)\, d^nx&=(2-\g)\int_{\bbR^n}|x|^{\g-4}f(x)[x\cdot(\nabla f)(x)]\, d^nx\\
&\quad-\int_{\bbR^n}|x|^{\g-2}|(\nabla f)(x)|^2\, d^nx.
\end{split}
\end{align}
Similarly, one obtains
\begin{align}
\begin{split}\lb{2.9}
\sum_{j,k=1}^n \int_{\bbR^n}|x|^{\g-4}f(x)x_jx_kf_{x_j,x_k}(x)\, d^nx&=(3-n-\g)\int_{\bbR^n}|x|^{\g-4}f(x)[x\cdot(\nabla f)(x)]\, d^nx\\
&\quad-\int_{\bbR^n}|x|^{\g-4}|x\cdot(\nabla f)(x)|^2\, d^nx.
\end{split}
\end{align}
Combining \eqref{2.6}-\eqref{2.9} and simplifying yields \eqref{2.1}.

To arrive at inequality \eqref{2.2}, note by Cauchy's inequality,
\begin{align}
-\int_{\bbR^n}|x|^{\g-4} |x\cdot(\nabla f)(x)|^2\, d^nx\geq-\int_{\bbR^n}|x|^{\g-2}|(\nabla f)(x)|^2\, d^nx,
\end{align}
one concludes as long as $\a(\a+2\g-4)\geq0$ the inequality \eqref{2.1} can be further estimated from below.
\end{proof}

\begin{remark} \lb{r2.2}
When $\g=0$, the inequalities \eqref{2.1} and \eqref{2.2} coincide with those found in \cite[Eq. (2.1), (2.2)]{GL18}, and hence contain all the special cases included therein as well.\hfill$\diamond$
\end{remark}

As a special case of \eqref{2.2} one obtains a power weighted Rellich inequality as follows:

\begin{corollary} \lb{c2.3}
Let $\g\in\bbR$ and $f \in C^{\infty}_0(\bbR^n \backslash \{0\})$, $n\in\bbN$, $n\geq2$. If $n\geq\g\geq2$, or if $n\geq 4-\g$ with $\g\leq2$, then,
\begin{align}
\begin{split}\lb{2.11}
\int_{\bbR^n} |x|^\g|(\Delta f)(x)|^2 \, d^n x 
& \geq [(n-\g) (\g+n - 4)/4]^2 \int_{\bbR^n} |x|^{\g-4} |f(x)|^2 \, d^n x\\
&= \bigg[\frac{(n-2)^2}{4}-\frac{(\g-2)^2}{4}\bigg]^2\int_{\bbR^n} |x|^{\g-4} |f(x)|^2 \, d^n x
\end{split}
\end{align} 
\end{corollary}
\begin{proof}
Choosing $\b=\a(n-\a-\g)/2$ in \eqref{2.2} results in
\begin{equation}
\int_{\bbR^n} |x|^\g|(\Delta f)(x)|^2 \, d^n x 
\geq F_{\g,n}(\alpha) \int_{\bbR^n} |x|^{\g-4} |f(x)|^2 \, d^n x,
\end{equation}
with
\begin{equation}
F_{\g,n}(\alpha) = \alpha (n - \alpha-\g) [(n - 4) (\alpha -2 ) -  (\alpha/2) (n - \alpha-\g)+\g(n+\a+\g-6)]/2. 
\end{equation}
Maximizing $F_{\g,n}(\alpha)$ with respect to $\alpha$ yields maxima at 
\begin{equation}
\a=\alpha_{\pm} = 2-\g \pm \big[\big((\g-2)^2+(n-2)^2\big)/2\big]^{1/2}.  
\end{equation}
Taking the constraint $\alpha(\a-4+2\g) \geq 0$ into account results in $(n-2)^2\geq(\g-2)^2$, or 
$n(n-4)\geq\g(\g-4)$, hence one concludes $n\geq\g$ is needed if $\g\geq2$ or $n\geq 4-\g$ if $\g\leq2$. The fact 
\begin{equation}
F_{\g,n}(\alpha_{\pm}) = [(n-\g) (\g+n - 4) /4]^2=16^{-1}\big[(n-2)^2-(\g-2)^2\big]^2,   
\end{equation}
then yields the weighted Rellich inequality \eqref{2.11}.
\end{proof}

\begin{remark} \lb{r2.4}
$(i)$ By setting $\g=0$ in Corollary \ref{c2.3}, one recovers Rellich's classical inequality for $n\geq5$.  \\[1mm]
$(ii)$ By setting $\g=-\a$, the condition $n\geq 4-\g$ for some $\g\leq2$ recovers Corollary 6.2.2 of \cite[p. 214]{BEL15} (see also \cite[Thm. 12]{DH98}, \cite[Thm. 3.3]{Mi00}, \cite{TZ07}).

The case $n\geq\g\geq2$ is a special case of \cite[Theorem 3.1]{CM12}, which is also the content of our Theorem \ref{t3.11}. In particular, by assuming that $4^{-1}(n-2)^2-4^{-1}(\g-2)^2 \geq 0$, one infers that the minimum in \eqref{3.45} occurs for $j=0$ so that the coefficient on the right-hand side of \eqref{3.44} coincides with that on the right-hand side of \eqref{2.11}.

See Remark \ref{r3.12} concerning the optimality of the constant in \eqref{2.11}
\hfill $\diamond$
\end{remark}

Inequalities \eqref{2.1} and \eqref{2.2} also imply the following power weighted Hardy--Rellich-type result:

\begin{corollary}\lb{c2.5}
Let $\g\in\bbR$ and $f \in C^{\infty}_0(\bbR^n \backslash \{0\})$, $n\in\bbN$, $n\geq2$. If $n\geq\g\geq2$, or if $n\geq 8-3\g$ with $\g\leq 2$, then,
\begin{align}
\int_{\bbR^n} |x|^\g|(\Delta f)(x)|^2 \, d^n x 
& \geq [(n-\g)/2]^2 \int_{\bbR^n} |x|^{\g-2} |(\nabla f)(x)|^2 \, d^n x.     \lb{2.16}
\end{align} 
In addition,
\begin{align}
\int_{\bbR^n} |x|^\g|(\Delta f)(x)|^2 \, d^n x 
& \geq 4(n-4-\g) \int_{\bbR^n} |x|^{\g-2} |(\nabla f)(x)|^2 \, d^n x,\quad n\geq4+\g,\ \g\geq0,    \lb{2.17}
\end{align}
and
\begin{align}
\int_{\bbR^n} |x|^\g|(\Delta f)(x)|^2 \, d^n x 
& \geq (4-2\g)(\g+n-4) \int_{\bbR^n} |x|^{\g-2} |(\nabla f)(x)|^2 \, d^n x,\quad n\geq 4-\g,\ \g\leq2.    \lb{2.18}
\end{align}
Furthermore, if $n\geq\g\geq2$, or if $n\geq 4-\g$ with $\g\leq2$, then,
\begin{align}
\int_{\bbR^n} |x|^\g|(\Delta f)(x)|^2 \, d^n x 
& \geq [(n-\g)/2]^2 \int_{\bbR^n} |x|^{\g-4} |x \cdot (\nabla f)(x)|^2 \, d^n x.    \lb{2.19}
\end{align} 
\end{corollary}
\begin{proof}
Again, we chose $f \in C^{\infty}_0(\bbR^n \backslash \{0\})$ real-valued for simplicity throughout this proof. Choosing $\beta = 0$ in 
\eqref{2.2} yields
\begin{equation}
\int_{\bbR^n} |x|^\g|(\Delta f)(x)|^2 \, d^n x 
\geq \alpha (n - \alpha-\g) \int_{\bbR^n} |x|^{\g-2} |(\nabla f)(x)|^2 \, d^n x.
\end{equation} 
Maximizing $G_{\g,n}(\alpha) = \alpha (n - \alpha-\g)$ with respect to $\alpha$ yields a 
maximum at $\alpha_1 = (n-\g)/2$. Once again, considering the constraint $\alpha(\a-4+2\g) \geq 0$, equivalently, $\alpha \geq 0$ if $\g\geq2$ or $\a\geq4-2\g$ if $\g\leq2$, and noting $G_{\g,n}(\a_1)=[(n-\g)/2]^2$ proves \eqref{2.16}. 

Choosing $\alpha = 4$, $\beta = 0$ for $\g\geq0$ in \eqref{2.2} yields \eqref{2.17} with coefficient $4(n-4-\g)$, and requiring $4(n-4-\g) \geq 0$, in order to get a nontrivial inequality, yields $n\geq4+\g$. In addition, choosing $\alpha = 4$, $\beta = 0$ for $\g\leq0$ in \eqref{2.1} (thus allowing the omission of the second integral in \eqref{2.1}) yields \eqref{2.17} with coefficient $4(\g+n-4)$ for $n\geq 4-\g$. Combining both of these, we obtain \eqref{2.17} with the constant $4(n-4-|\g|)$, instead of $4(n-4-\g)$, on the right-hand side for $n\geq 4+|\g|$, $\g\in\bbR$. (See the proof of \eqref{2.18} below for omitting $\g<0$ in \eqref{2.17}.) 

We can also accomplish the slightly different inequality \eqref{2.18} by choosing $\a=4-2\g,$ $\b=0$ in \eqref{2.1} so that the coefficient is now $(4-2\g)(\g+n-4)$, thus we need $n\geq 4-\g$ if $\g\leq2$ and $n\leq 4-\g$ if $\g\geq2$ (the latter being thrown out) for this to be positive.

Note that inequality \eqref{2.18} is superior to \eqref{2.17} for all $\g\leq0$ and when $n\leq8-\g$ for $0\leq\g\leq2$. Furthermore, \eqref{2.18} has a smaller lower bound on the dimension, $n$, than \eqref{2.17} for all $0\leq\g\leq2$, and the inequalities are the same when $\g=0$, hence why \eqref{2.17} is written only for $\g\geq0$ in the statement of the corollary.

The choice $\b=(n-4)(\a-2)+\g(n+\a+\g-6)$  in \eqref{2.1} results in
\begin{align}
\no  \int_{\bbR^n} |x|^\g|(\Delta f)(x)|^2 \, d^n x 
& \geq [(n - 4)(4 - \alpha)-\g(\a+2\g+2n-12)] \int_{\bbR^n} |x|^{\g-2} |(\nabla f)(x)|^2 \, d^n x  \\
& \quad - \alpha (\alpha - 4+2\g) \int_{\bbR^n} |x|^{\g-4} |x \cdot (\nabla f)(x)|^2 \, d^n x.   
\lb{2.21}
\end{align}
Next, one studies when the coefficient of the first integral on the right-hand side of \eqref{2.21} is greater than or equal to zero (to apply Cauchy's inequality), which is equivalent to
\begin{align}\lb{2.22a}
-2\g^2-(\a+2n-12)\g+(n-4)(4-\a)\geq0.
\end{align}
The discriminant of the left-hand side of \eqref{2.22a} can be written $(\a-2(n-2))^2$, so that, assuming $\a\leq2(n-2)$, one sees that \eqref{2.22a} holds when $4-2\g\geq\a$ and $n\geq 4-\g$. Moreover, the requirement $n\geq 4-\g$ is equivalent to $2(n-2)\geq 4-2\g$, hence $4-2\g\geq\a$ implies the assumption $\a\leq2(n-2)$ already. Therefore, for $n \geq 4-\g$ and $\alpha \leq (4-2\g)$, applying Cauchy's inequality to the 1st term on the right-hand side of \eqref{2.21} now yields
\begin{equation}
\int_{\bbR^n} |x|^\g|(\Delta f)(x)|^2 \, d^n x 
\geq H_{\g,n}(\alpha) \int_{\bbR^n} |x|^{\g-4} |x \cdot (\nabla f)(x)|^2 \, d^n x,    
\end{equation}
where $H_{\g,n}(\alpha) = -\a^2+(8-3\g-n)\a-2\g^2+12\g-2\g n+4n-16$. Requiring $H_{\g,n}(\alpha)\geq0$ and maximizing $H_{\g,n}$ subject to the constraints $(4 - n-\g) \leq \alpha \leq (4-2\g)$ yields a maximum at $\alpha_2 = (8 - n-3\g)/2$. These constraints applied to $\a_2$ further yield the requirement $n\geq\g$ in addition to the previous requirement $n\geq 4-\g$. These requirements can be equivalently written $n\geq\g$ with $\g\geq 2$ (as this implies $n\geq 4-\g)$ or $n\geq 4-\g$ with $\g\leq2$ (as this implies $n\geq\g$). Finally, noting
$H_{\g,n}(\a_2) = [(n-\g)/2]^2$ implies \eqref{2.19}.
\end{proof}

\begin{remark} \lb{r2.6}
A sequence of extensions of \eqref{2.16} on bounded domains containing $0$, valid for $n \geq 5$, was derived by Tertikas and Zographopoulos \cite[Theorem~1.7]{TZ07}. For $\g=0$, an extension 
of inequality \eqref{2.16} valid for $n = 4$ and for bounded open domains containing $0$ was proved by \cite[Theorem~2.1\,(b)]{AGS06}. Moreover, an alternative weighted inequality whose special cases imply inequality \eqref{2.16} (under different parameter restrictions) appeared in \cite{Co09}. 

While the constant $(n-\g)^2/4$ in \eqref{2.16} is known to be optimal (cf.\ \cite{Co09} or \cite{TZ07}), the constant $4(n-4-\g)$ in \eqref{2.17} is not, the sharp constant also being $(n-\g)^2/4$ (cf.\ \cite{Co09}). 
\hfill $\diamond$ 
\end{remark}

The following special case of inequality \eqref{2.2} is an extension of Schmincke's inequality \cite{Sc72} to the power-weighted case:

\begin{theorem} \lb{t2.7}
Let $\g\in\bbR$ and $f\in C_0^\infty(\bbR^n\backslash\{0\})$, $n\in\bbN$, $n\geq2$. Then,
\begin{align} 
\int_{\bbR^n} |x|^\g|(\Delta f)(x)|^2 \, d^n x 
& \geq - s \int_{\bbR^n} |x|^{\g-2} |(\nabla f)(x)|^2 \, d^n x  \no \\
& \quad + [(\g+n - 4)/4]^2 \big[(\g-n)^2+4s\big] \int_{\bbR^n} |x|^{\g-4} |f(x)|^2 \, d^n x,   \lb{2.24} \\
& \hspace*{3.1cm} s \in \big[-2^{-1}\big\{(n-2)^2-(\g-2)^2\big\}, \infty\big).    \no 
\end{align} 
\end{theorem}
\begin{proof}
The choice $\beta = 2^{-1}(\g+n-4)\big[\alpha+\g -2 - 2^{-1}(\g+n-4)\big]$, and the introduction of the new variable
\begin{equation}
s = s_{n,\g}(\alpha) = \alpha^2+2(\g-2)\a + 2^{-1}\big[(\g-2)^2-(n-2)^2\big],
\end{equation}
renders \eqref{2.2} into the  weighted Schmincke inequality \eqref{2.24}.
Here the requirement $\a(\a-4+2\g)=\a^2+2(\g-2)\a\geq0$, immediately yields the range requirement for $s$.
Note that the inequality is trivially satisfied for the choice $n=4-\g$ since, in this case, $s\geq0$ with the second coefficient equal to 0.
\end{proof}

\begin{remark} \lb{r2.8}
Inequality \eqref{2.24} 
is precisely the content of Corollary~6 in Bennett \cite{Be89}, in particular, \eqref{2.2} thus recovers Bennett's result which is itself an extension of Schmincke's original inequality. However, Bennett's result requires the additional restriction $\g\in(-\infty,2)\cup(n,\infty)$. Hence \eqref{2.24} is an extension of Bennett's result to all $\g\in\bbR$.

Moreover, assuming $-2^{-1}\big\{(n-2)^2-(\g-2)^2\big\} \leq 0$ (equivalently $n(n-4)\geq \g(\g-4)$) permits the value $s=0$ and hence implies Rellich's inequality \eqref{2.11} under the same assumptions as in Corollary \ref{c2.3}.

If $-2^{-1}\big\{(n-2)^2-(\g-2)^2\big\} \leq- (\g-n)^2/4$ (equivalently, $n\geq\g$ for some $\g\geq2$, or,  $n\geq8-3\g$ for some $\g\leq2$), the value 
$s= - (\g-n)^2/4$ is permitted, yielding inequality \eqref{2.16} under the same assumptions as in Corollary \ref{c2.5}.

Finally, if $\g\geq n$ with $\g\geq2$ or $8-3\g\geq n$ with $\g\leq2$, one has $(\g-n)^2+4s\geq0$ for all $s \in \big[-2^{-1}\big\{(n-2)^2-(\g-2)^2\big\}, \infty\big)$. Hence, one can choose $s=-2^{-1}\big\{(n-2)^2-(\g-2)^2\big\}$ in \eqref{2.24} to conclude
\begin{align} 
\begin{split}
\int_{\bbR^n} |x|^\g|(\Delta f)(x)|^2 \, d^n x 
\geq 2^{-1}\big\{(n-2)^2-(\g-2)^2\big\} \int_{\bbR^n} |x|^{\g-2} |(\nabla f)(x)|^2 \, d^n x,&  \lb{2.26} \\
\g\geq n, \; \g\geq2, \; \text{or,} \; 8-3\g\geq n,\; \g\leq2.& 
\end{split} 
\end{align} 
When comparing \eqref{2.26} to the previous results given in \eqref{2.17} and \eqref{2.18}, different choices of $\g$ and $n$ (being mindful of parameter requirements) yield that none of the inequalities are superior for all choices.
In particular, directly comparing the constants given in \eqref{2.17}, and \eqref{2.26} allows one to conclude that \eqref{2.26} will be strictly superior provided that each are valid and
\begin{equation}\lb{2.27}
(n-8)(n-4)>\g(\g-12).
\end{equation}
Similarly, \eqref{2.26} is strictly superior to \eqref{2.18} provided both are valid and
\begin{equation}\lb{2.28}
(\g+n-4)(3\g+n-8)\geq0.
\end{equation}
\hfill $\diamond$ 
\end{remark}

We conclude this section with a few additional remarks regarding our results.

\begin{remark} \lb{r2.9} 
The factorization approach was employed in the context of the classical Hardy inequality in \cite{GP80} (and some of its logarithmic refinements in \cite{Ge84}). We now illustrate how this approach applies to a weighted Hardy inequality. Given $\alpha,\g \in \bbR$, $n \in \bbN$, $n \geq 2$, one introduces the two-parameter family of homogeneous vector-valued differential expressions 
\begin{equation}
T_{\alpha,\g} := |x|^{\g/2}\big[\nabla + \alpha |x|^{-2} x\big], \quad  x \in \bbR^n \backslash \{0\},
\end{equation}
with formal adjoint, denoted by $T_{\alpha,\g}^+$, 
\begin{equation}
T_{\alpha,\g}^+ = |x|^{\g/2}\big[- {\rm div(\, \cdot \,)} + \big(\alpha-2^{-1}\g\big) |x|^{-2} x \, \cdot\big], \quad  
\; x \in \bbR^n \backslash \{0\},    
\end{equation}
such that 
\begin{equation}
T_{\alpha,\g}^+ T_{\alpha,\g} = - |x|^\g\Delta-\g|x|^{\g-2}x\cdot\nabla + \alpha (\alpha + 2 - n-\g) |x|^{\g-2}. 
\end{equation}
Thus, for $f \in C_0^{\infty}(\bbR^n \backslash \{0\})$, 
\begin{align}
\begin{split}
0  \leq &\int_{\bbR^n} |(T_{\alpha,\g} f)(x)|^2 \, d^n x 
= \int_{\bbR^n} \ol{f(x)} (T_{\alpha,\g}^+ T_{\alpha,\g} f)(x) \, d^n x     \\
=& -\int_{\bbR^n} |x|^\g\ol{f(x)}(\Delta f)(x) \, d^n x-\g\int_{\bbR^n} |x|^{\g-2}\ol{f(x)}[x\cdot(\nabla f)(x)] \, d^n x\\
&+ \alpha (\alpha + 2 - n-\g)
\int_{\bbR^n} |x|^{\g-2} |f(x)|^2 \, d^n x.  \lb{2.32} 
\end{split} 
\end{align}
By standard integration by parts,
\begin{align}
-\int_{\bbR^n} |x|^\g\ol{f(x)}(\Delta f)(x) \, d^n x=\int_{\bbR^n} |x|^\g|(\nabla f)(x)|^2 \, d^n x+\g\int_{\bbR^n} |x|^{\g-2}\ol{f(x)}[x\cdot(\nabla f)(x)] \, d^n x
\end{align}
and hence substituting into \eqref{2.32} and rearranging terms yields
\begin{equation}
\int_{\bbR^n} |x|^\g|(\nabla f)(x)|^2 \, d^n x \geq \alpha (n-2+\g-\a)
\int_{\bbR^n} |x|^{\g-2} |f(x)|^2 \, d^n x.   
\end{equation}
Maximizing $\alpha (n-2+\g-\a)$ with respect to $\alpha$ yields the weighted  Hardy inequality,
\begin{align}
\begin{split}
\int_{\bbR^n} |x|^\g|(\nabla f)(x)|^2 \, d^n x \geq [(n - 2+\g)/2]^2 
\int_{\bbR^n} |x|^{\g-2} |f(x)|^2 \, d^n x,&   \\
f \in C_0^{\infty}(\bbR^n \backslash \{0\}), \; n \geq 2, \; \gamma \in\bbR \backslash \{2-n\}.&  
\lb{2.35} 
\end{split}
\end{align}
Again, it is well-known that the constant in \eqref{2.35} is optimal (cf., e.g., \cite{Mu14b}).

One should note that applying \eqref{2.35} to the r.h.s. of \eqref{2.16} yields \eqref{2.11} but with more constraints on $n$ when $\g\leq2$. 
\hfill $\diamond$
\end{remark}

In fact, the factorization approach also yields a known improvement of the weighted Hardy inequality (see, e.g., \cite[Theorem~1.2.5]{BEL15}, specializing it to $p=2$, $\varepsilon =\g/2$). We now sketch the corresponding argument. 

\begin{remark} \lb{r2.10}
Given $\alpha,\g \in \bbR$, $n \in \bbN$, $n \geq 2$, one introduces the following modified two-parameter family of homogeneous vector-valued differential expressions
\begin{equation}
\wti T_{\alpha,\g} := |x|^{\g/2}\big[\big(|x|^{-1} x\big) \cdot \nabla + \alpha |x|^{-1}\big], 
\quad  x \in \bbR^n \backslash \{0\},  
\end{equation}
with formal adjoint, denoted by $\big(\wti T_{\alpha,\g}\big)^+$, 
\begin{equation}
\big(\wti T_{\alpha,\g}\big)^+= |x|^{\g/2}\big[- \big(|x|^{-1} x\big) \cdot \nabla  
+ \big(\alpha - n +1-2^{-1}\g\big) |x|^{-1}\big], \quad  \; x \in \bbR^n \backslash \{0\}.  
\end{equation}
Exploiting the identity (for $f \in C_0^{\infty}(\bbR^n \backslash \{0\})$, for simplicity) with $x \in \bbR^n \backslash \{0\}$, 
\begin{align}
& \big[|x|^{\g/2-1} x \cdot \nabla\big] \big[|x|^{\g/2-1} x \cdot (\nabla f)(x)\big] = 
|x|^{\g-2}2^{-1}\g[x\cdot(\nabla f)(x)]+|x|^{\g-2} \sum_{j,k=1}^n x_j x_k f_{x_j, x_k}(x), 
\end{align}
one computes 
\begin{align}
\begin{split}
\big(\wti T_{\alpha,\g}\big)^+ \wti T_{\alpha,\g} 
=& - |x|^{\g-2} \sum_{j,k=1}^n x_j x_k \partial_{x_j} \partial_{x_k} 
- (n-1+\g) |x|^{\g-2} [x \cdot (\nabla f)(x)]\\
&+ \alpha (\alpha + 2 - n-\g) |x|^{\g-2},   \quad x \in \bbR^n \backslash \{0\}.    
\end{split}
\end{align}
Thus, appropriate integration by parts yield 
\begin{align}
0  \leq& \int_{\bbR^n} \big|\big(\wti T_{\alpha,\g} f\big)(x)\big|^2 \, d^n x 
= \int_{\bbR^n} \ol{f(x)} \big(\big(\wti T_{\alpha,\g}\big)^+ \wti T_{\alpha,\g} f\big)(x) \, d^n x   \no \\
=& - \int_{\bbR^n} |x|^{\g-2} \sum_{j,k=1}^n x_j x_k \ol{f(x)} f_{x_j, x_k}(x)\, d^n x - (n-1+\g)\int_{\bbR^n} |x|^{\g-2} \ol{f(x)} [x \cdot (\nabla f)(x)]\, d^n x    \no \\
& + \alpha (\alpha - n + 2-\g) \int_{\bbR^n} |x|^{\g-2} |f(x)|^2 \, d^n x,   \no \\
=& \int_{\bbR^n} |x|^{\g-2} \bigg\{|x \cdot (\nabla f)(x)|^2 
+ \alpha (\alpha - n + 2-\g)|f(x)|^2\bigg\}\, d^n x, \quad f \in C_0^{\infty}(\bbR^n \backslash \{0\}).     
\end{align}
Here we used 
\begin{align}
\begin{split}
& \sum_{j,k=1}^n \int_{\bbR^n} |x|^{\g-2} x_j x_k \ol{f(x)} f_{x_j, x_k}(x) \, d^n x 
= - \int_{\bbR^n} |x|^{\g-2} |x \cdot (\nabla f)(x)|^2 \, d^n x    \\
& \quad - (n-1+\g) \int_{\bbR^n} |x|^{\g-2} \ol{f(x)} [x \cdot (\nabla f)(x)] \, d^n x, 
\quad f \in C_0^{\infty}(\bbR^n \backslash \{0\}).   
\end{split}
\end{align}
Thus,
\begin{equation}
\int_{\bbR^n} |x|^\g\big| |x|^{-1} x \cdot \nabla f(x)\big|^2 \, d^n x \geq \alpha (n-2+\g-\a)
\int_{\bbR^n} |x|^{\g-2} |f(x)|^2 \, d^n x.    
\end{equation}
Maximizing $\alpha (n-2+\g-\a)$ with respect to $\alpha$ yields the improved Hardy inequality,
\begin{align}
\begin{split}
\int_{\bbR^n} |x|^\g\big||x|^{-1} x \cdot \nabla f(x)\big|^2 \, d^n x \geq  [(n - 2+\g)/2]^2 
\int_{\bbR^n} |x|^{\g-2} |f(x)|^2 \, d^n x,&     \\
 f \in C_0^{\infty}(\bbR^n \backslash \{0\}), \; n \geq 2, \; \gamma \in\bbR \backslash \{2-n\}.&   \lb{2.43} 
\end{split}
\end{align}
(By Cauchy's inequality, \eqref{2.43} implies the weighted Hardy inequality \eqref{2.35}.)
Again, it is known that the constant in \eqref{2.43} is optimal (cf., e.g., \cite[Theorem~1.2.5]{BEL15}). 
\hfill $\diamond$
\end{remark}

\section{Factorization With an Additional Spherical Term}\lb{s3}

In this section, we extend Theorem \ref{t2.1} by introducing a spherical term into the computation by means of the Laplace--Beltrami operator. Before stating our main result, we introduce some notation:

Let $\mathbb{S}^{n-1}$ denote the $(n-1)$-dimensional unit sphere in $\bbR^n,\ n\in\bbN,\ n\geq 2,$ with $d^{n-1}\omega:=d^{n-1}\omega(\theta)$ the usual volume measure on $\mathbb{S}^{n-1}$. We denote by $-\Delta_{\mathbb{S}^{n-1}}$ the nonnegative, self-adjoint Laplace--Beltrami operator in $L^2(\mathbb{S}^{n-1};d^{n-1}\omega)$ and we recall that in terms of polar coordinates, 
\begin{align}
- \Delta = - r^{1-n} \f{\partial}{\partial r} r^{n-1} \f{\partial}{\partial r} - \f{1}{r^2} \Delta_{\bbS^{n-1}}  
= - \f{\partial^2}{\partial r^2} - \f{n-1}{r} \f{\partial}{\partial r} - \f{1}{r^2} \Delta_{\bbS^{n-1}}.
\end{align}
Our extension of Theorem \ref{t2.1} can then be stated as follows:

\begin{theorem} \lb{t3.1}
Let $\a,\b,\g,\tau\in\bbR$ and $f\in C_0^\infty(\bbR^n\backslash\{0\})$, $n\in\bbN$, $n\geq2$. Then,
\begin{align} 
& \int_{\bbR^n}|x|^\g |(\Delta f)(x)|^2\, d^nx\geq [\a(\g+n-4)-2\b]\int_{\bbR^n}|x|^{\g-2}|(\nabla f)(x)|^2\, d^nx \no \\
&\quad-\a(\a-4+2\g)\int_{\bbR^n}|x|^{\g-4}|x\cdot(\nabla f)(x)|^2\, d^nx   \no \\
&\quad+\b[(n-4)(\a-2)-\b+\g(n+\a+\g-6)]\int_{\bbR^n}|x|^{\g-4}|f(x)|^2 d^nx\no  \\
&\quad-\tau[(\g+n-4)(2-\a-\g)+2\b]  \lb{3.1} \\
&\quad\qquad \times \int_0^\infty\int_{{\bbS^{n-1}}}r^{\g-4}\big|\big((-\Delta_{\bbS^{n-1}})^{1/2}f\big)(r,\t)\big|^2\,  d^{n-1} \omega(\theta) \, r^{n-1}dr  \no  \\
&\quad -2\tau\int_0^\infty\int_{{\bbS^{n-1}}}r^{\g-2}\bigg|\bigg((-\Delta_{\bbS^{n-1}})^{1/2}\frac{\partial f}{\partial r}\bigg)(r,\t)\bigg|^2\,  d^{n-1} \omega(\theta) \, r^{n-1}dr   \no  \\
&\quad -\tau(\tau+2)\int_0^\infty\int_{{\bbS^{n-1}}}r^{\g-4}|(\Delta_{\bbS^{n-1}} f)(r,\t)|^2\, d^{n-1} \omega(\theta) \, r^{n-1}dr.  \no 
\end{align}

In addition, if $\a(\a-4+2\g)\geq0$, then,
\begin{align} 
& \int_{\bbR^n}|x|^\g|(\Delta f)(x)|^2\, d^nx \geq [\a(n-\a-\g)-2\b]\int_{\bbR^n}|x|^{\g-2}|(\nabla f)(x)|^2\, d^nx  \no \\
&\quad+\b[(n-4)(\a-2)-\b+\g(n+\a+\g-6)]\int_{\bbR^n}|x|^{\g-4}|f(x)|^2 d^nx   \no \\
&\quad-\tau[(\g+n-4)(2-\a-\g)+2\b]  \no \\
&\quad\qquad \times \int_0^\infty\int_{{\bbS^{n-1}}}r^{\g-4}\big|\big((-\Delta_{\bbS^{n-1}})^{1/2}f\big)(r,\t)\big|^2\,  d^{n-1} \omega(\theta) \, r^{n-1}dr  \no  \\
&\quad -2\tau\int_0^\infty\int_{{\bbS^{n-1}}}r^{\g-2}\bigg|\bigg((-\Delta_{\bbS^{n-1}})^{1/2}\frac{\partial f}{\partial r}\bigg)(r,\t)\bigg|^2\,  d^{n-1} \omega(\theta) \, r^{n-1}dr   \no  \\
&\quad -\tau(\tau+2)\int_0^\infty\int_{{\bbS^{n-1}}}r^{\g-4}|(\Delta_{\bbS^{n-1}} f)(r,\t)|^2\, d^{n-1} \omega(\theta) \, r^{n-1}dr.  \lb{3.2} 
\end{align}
\end{theorem}

\begin{remark}\lb{r3.2}
Clearly, choosing $\tau=0$ in Theorem \ref{t3.1} recovers Theorem \ref{t2.1}. Furthermore, one notes that choosing $\tau\in(-2,0)$ and assuming $(\g+n-4)(2-\a-\g)+2\b>0$ in \eqref{3.1} and \eqref{3.2}, all three of the new $\tau$-dependent terms will have positive coefficients. This choice of $\tau\in(-2,0)$ will be evident in many of the applications of Theorem \ref{t3.1} in this section.
\hfill$\diamond$
\end{remark}

Before proving Theorem \ref{t3.1}, we recall some standard facts and prove a few needed general results.
For $j\in\bbN_0= \bbN \cup \{0\}$, let
\begin{equation}\lb{3.3}
\lambda_j=j(j+n-2),\quad j\in\bbN_0,
\end{equation}
be the eigenvalues of $-\Delta_{\bbS^{n-1}}$, that is, $\sigma (-\Delta_{\bbS^{n-1}}) = \{j(j+n-2)\}_{j \in \bbN_0}$, of multiplicity
\begin{equation}
m(\lambda_j)=(2j+n-2)(j+n-2)^{-1} {j+n-2 \choose n-2},\quad j\in\bbN_0,
\end{equation}
with corresponding eigenfunctions $\varphi_{j,\ell},\ j\in\bbN_0,\ \ell\in\{1,\dots,m(\lambda_j)\}$. We may (and will) assume that $\{\varphi_{j,\ell}\}_{j\in\bbN_0,\ \ell\in\{1,\dots,m(\lambda_j)\}}$ is an orthonormal basis of $L^2({\bbS^{n-1}};d^{n-1}\omega)$, and let
\begin{align}
\begin{split}
F_{f,j,\ell}(r)=(\varphi_{j,\ell},f(r,\dott))_{L^2({\bbS^{n-1}};d^{n-1}\omega)}=\int_{\bbS^{n-1}} \overline{\varphi_{j,\ell}(\theta)}  f(r,\theta) \, d^{n-1}\omega(\theta),&\\
f\in C_0^\infty(\bbR^n\backslash\{0\}),\quad r>0,\quad j\in\bbN_0,\ \ell\in\{1,\dots,m(\lambda_j)\}.&
\end{split}
\end{align}

Given $\a,\b,\g,\tau\in\bbR$ and $n\in\bbN,\ n\geq2$, we now introduce the four-parameter $n$-dimensional differential expression
\begin{equation}
T_{\a,\b,\g,\tau}=T_{\a,\b,\g}+\tau |x|^{\frac{\g}{2}-2}(-\Delta_{\bbS^{n-1}}),\quad x\in\bbR^n\backslash\{0\}, 
\lb{3.7}
\end{equation}
and its formal adjoint, denoted by $T_{\a,\b,\g,\tau}^+$,
\begin{equation}
T_{\a,\b,\g,\tau}^+=T_{\a,\b,\g}^+ + \tau|x|^{\frac{\g}{2}-2}(-\Delta_{\bbS^{n-1}}),\quad x\in\bbR^n\backslash\{0\}.
\end{equation}
In polar coordinates, one has for $r\in (0,\infty)$,
\begin{align}
T_{\a,\b,\g,\tau}=r^{\g/2}\bigg\{-r^{-(n-1)}\frac{\partial}{\partial r}\bigg(r^{n-1}\pr\bigg)+\a r^{-1}\pr+\b r^{-2}+(1+\tau)r^{-2}(-\Delta_{\bbS^{n-1}})\bigg\}, 
\end{align}
and
\begin{align}
T_{\a,\b,\g,\tau}^+=r^{\g/2}\bigg\{&-r^{-(n-1)}\frac{\partial}{\partial r}\bigg(r^{n-1}\pr\bigg)-(\a+\g) r^{-1}\pr\\
&+\big[\b-\a\big(n-2\big)-\big(2^{-1}\g+\a+n-2\big)2^{-1}\g\big]r^{-2}+(1+\tau)r^{-2}(-\Delta_{\bbS^{n-1}})\bigg\}.   \no 
\end{align}
In direct analogy to the proof of Theorem \ref{t2.1}, we now study the composition $T^+_{\a,\b,\g,\tau}T_{\a,\b,\g,\tau}$. We first prove a few technical lemmas that will be used to derive identities and estimates employed in the proof of Theorem \ref{t3.1}. The first result shows that the square root of the Laplace--Beltrami operator nearly commutes with the radial partial derivative.

\begin{lemma}\lb{l3.3}
Let $f\in C_0^\infty(\bbR^n\backslash\{0\})$, $n\in\bbN$, $n\geq2$. Then for all $r_0 \in (0,\infty)$ and $k\in\bbN$, one has 
\begin{equation}
\frac{d^k}{d r^k}\big((-\Delta_{\bbS^{n-1}})^{1/2} f(r,\dott)\big)(r_0)=(-\Delta_{\bbS^{n-1}})^{1/2}\bigg(\frac{\partial^k f}{\partial r^k}(r_0,\dott)\bigg)\in L^2({\bbS^{n-1}};d^{n-1}\omega).
\end{equation}
In particular, for $r\in(0,\infty)$, the function
\begin{equation}
r\mapsto (-\Delta_{\bbS^{n-1}})^{1/2}f(r,\dott)\in L^2({\bbS^{n-1}};d^{n-1}\omega)
\end{equation}
lies in $C_0^\infty\big((0,\infty),L^2({\bbS^{n-1}};d^{n-1}\omega)\big)$.
\end{lemma}
\begin{proof}
First, we note that $-\Delta_{\bbS^{n-1}}$ and $\partial/\partial r$ commute on $C_0^{\infty} (\bbR^n\backslash\{0\})$, that is, for $f\in C_0^{\infty}(\bbR^n\backslash\{0\})$,
\begin{equation}
\pr(-\Delta_{\bbS^{n-1}} f)=-\Delta_{\bbS^{n-1}}\bigg(\frac{\partial f}{\partial r}\bigg)\in C_0^{\infty} (\bbR^n\backslash\{0\}).
\end{equation}
Hence, for $r_0 \in (0,\infty)$ and $f\in C_0^{\infty}(\bbR^n\backslash\{0\})$ real-valued (w.l.o.g., see Remark \ref{r1.1}),
\begin{align}
& \lim_{\varepsilon\downarrow 0}\varepsilon^{-1}\{(-\Delta_{\bbS^{n-1}} f)(r_0+\varepsilon,\dott)-(-\Delta_{\bbS^{n-1}} f)(r_0,\dott)\}  \no   \\
& \quad =\lim_{\varepsilon\downarrow 0}\varepsilon^{-1}\bigg\{\sum_{j\in\bbN_0}\sum_{\ell=1}^{m(\lambda_j)}\lambda_j F_{f,j,\ell}(r_0+\varepsilon)\varphi_{j,\ell}-\sum_{j\in\bbN_0}\sum_{\ell=1}^{m(\lambda_j)}\lambda_j F_{f,j,\ell}(r_0)\varphi_{j,\ell}\bigg\} \no \\
& \quad =\sum_{j\in\bbN_0}\sum_{\ell=1}^{m(\lambda_j)}\lambda_j F_{\frac{\partial f}{\partial r},j,\ell}(r_0)\varphi_{j,\ell},
\end{align}
therefore,
\begin{equation}
\lim_{\varepsilon\downarrow 0}\sum_{j\in\bbN_0}\sum_{\ell=1}^{m(\lambda_j)}\lambda_j^2 \big|\varepsilon^{-1}(F_{f,j,\ell}(r_0+\varepsilon)-F_{f,j,\ell}(r_0))-F_{\frac{\partial f}{\partial r},j,\ell}(r_0)\big|^2=0.
\end{equation}
As $\lambda_j\geq1,\ j\geq 1,$ this shows that
\begin{equation}
\lim_{\varepsilon\downarrow 0}\sum_{j\in\bbN_0}\sum_{\ell=1}^{m(\lambda_j)}\lambda_j\big|\varepsilon^{-1}(F_{f,j,\ell}(r_0+\varepsilon)-F_{f,j,\ell}(r_0))-F_{\frac{\partial f}{\partial r},j,\ell}(r_0)\big|^2=0,
\end{equation}
implying  
\begin{equation}\lb{3.16}
\lim_{\varepsilon \downarrow 0}\bigg\|\sum_{j\in\bbN_0}\sum_{\ell=1}^{m(\lambda_j)}\lambda_j^{1/2}\big[\varepsilon^{-1}(F_{f,j,\ell}(r_0+\varepsilon)-F_{f,j,\ell}(r_0))-F_{\frac{\partial f}{\partial r},j,\ell}(r_0)\big]\varphi_{j,\ell} \bigg\|^2_{L^2({\bbS^{n-1}};d^{n-1}\omega)} = 0.
\end{equation}
Thus, for $r_0 \in (0,\infty)$ and $f\in C_0^{\infty}(\bbR^n\backslash\{0\})$,
\begin{align}
& \frac{d}{d r}\big((-\Delta_{\bbS^{n-1}})^{1/2}f(r,\dott)\big)(r_0)=\lim_{\varepsilon\downarrow 0}\varepsilon^{-1}\{(-\Delta_{\bbS^{n-1}})^{1/2}f(r_0+\varepsilon,\dott)-(-\Delta_{\bbS^{n-1}})^{1/2}f(r_0,\dott)\}  \no   \\
& \quad =\lim_{\varepsilon\downarrow 0}\sum_{j\in\bbN_0}\sum_{\ell=1}^{m(\lambda_j)}\lambda_j^{1/2}\big[\varepsilon^{-1}( F_{f,j,\ell}(r_0+\varepsilon)-F_{f,j,\ell}(r_0)) 
-F_{\frac{\partial f}{\partial r},j,\ell}(r_0)\big]\varphi_{j,\ell}     \no \\
&\hspace*{3.6cm} 
+\sum_{j\in\bbN_0}\sum_{\ell=1}^{m(\lambda_j)}\lambda_j^{1/2} F_{\frac{\partial f}{\partial r},j,\ell}(r_0)\varphi_{j,\ell} \no \\
& \quad =\lim_{\varepsilon\downarrow 0}\sum_{j\in\bbN_0}\sum_{\ell=1}^{m(\lambda_j)}\lambda_j^{1/2}\big[\varepsilon^{-1}( F_{f,j,\ell}(r_0+\varepsilon)-F_{f,j,\ell}(r_0)) \\
&\hspace*{3.6cm}-F_{\frac{\partial f}{\partial r},j,\ell}(r_0)\big]\varphi_{j,\ell}+(-\Delta_{\bbS^{n-1}})^{1/2}\bigg(\frac{\partial f}{\partial r}(r_0,\dott)\bigg). \no 
\end{align}
Hence, applying \eqref{3.16}, one concludes that
\begin{equation}
\frac{d}{d r}\big(\big((-\Delta_{\bbS^{n-1}})^{1/2}f\big)(r,\dott)\big)(r_0)=(-\Delta_{\bbS^{n-1}})^{1/2}\bigg(\frac{\partial f}{\partial r}(r_0,\dott)\bigg).
\end{equation}
The lemma now follows by induction on $k \in \bbN$.
\end{proof}

\begin{lemma}\lb{l3.4}
Let $g \in C^{\infty}_0((0, \infty), L^2(\mathbb{S}^{n-1}, d^{n-1}\omega))$ and supposed that for all $r > 0,\ g(r)$ is a real-valued function on $\mathbb{S}^{n-1}$. Then, for $n \in \mathbb{N},\ n\geq 2$, one has  
\begin{align}
& - \int_0^\infty \big(r^{-(n-1)} (d/dr)\big(r^{n-1}(d/dr)\big(r^{-2}g\big)\big),g\big)_{L^2(\mathbb{S}^{n-1}, d^{n-1}\omega)} r^{n-1}dr     \lb{3.19} \\
& \quad =\int_0^\infty r^{-2}(g',g')_{L^2(\mathbb{S}^{n-1}, d^{n-1}\omega)}\, r^{n-1}dr+(n-4)\int_0^\infty r^{-4}(g,g)_{L^2(\mathbb{S}^{n-1}, d^{n-1}\omega)}\, r^{n-1}dr.   \no 
\end{align}
\end{lemma}
\begin{proof}
For $f,g\in C^\infty((0,\infty), L^2(\mathbb{S}^{n-1}, d^{n-1}\omega))$, the product rule
\begin{equation}
\frac{d}{d r}( f, g)_{L^2(\mathbb{S}^{n-1}, d^{n-1}\omega)}=( f', g)_{L^2(\mathbb{S}^{n-1}, d^{n-1}\omega)}+( f, g')_{L^2(\mathbb{S}^{n-1}, d^{n-1}\omega)}
\end{equation}
implies that for $f, g \in C_{0}^{\infty}((0, \infty), L^2(\mathbb{S}^{n-1}, d^{n-1}\omega))$ one has
\begin{equation}
\int_{0}^{\infty}( f', g)_{L^2(\mathbb{S}^{n-1}, d^{n-1}\omega)} \,d r=-\int_{0}^{\infty}( f, g')_{L^2(\mathbb{S}^{n-1}, d^{n-1}\omega)} \,d r.
\end{equation}
Thus, if $g\in C_0^\infty((0,\infty), L^2(\mathbb{S}^{n-1}, d^{n-1}\omega))$ and if for all $r > 0,\ g(r)$ is a real-valued function on $\mathbb{S}^{n-1}$, then
\begin{align}
& - \int_{0}^{\infty}\big(r^{-(n-1)} (d/dr)\big(r^{n-1} (d/dr)\big(r^{-2} g\big)\big), g\big)_{L^2(\mathbb{S}^{n-1}, d^{n-1}\omega)}\, r^{n-1} \,d r \no \\
& \quad = \int_{0}^{\infty}\big((d/dr)\big(r^{-2} g\big), g'\big)_{L^2(\mathbb{S}^{n-1}, d^{n-1}\omega)}\, r^{n-1} d r \no \\
& \quad = - 2\int_{0}^{\infty} r^{-3}( g, g')_{L^2(\mathbb{S}^{n-1}, d^{n-1}\omega)}\, r^{n-1} d r+\int_{0}^{\infty} r^{-2}( g', g')_{L^2(\mathbb{S}^{n-1}, d^{n-1}\omega)}\, r^{n-1} d r  \\
& \quad = - \int_{0}^{\infty} r^{-3}(d/dr)[(g, g)_{L^2(\mathbb{S}^{n-1}, d^{n-1}\omega)}]\, r^{n-1} d r+\int_{0}^{\infty} r^{-2}(g^{\prime}, g^{\prime})_{L^2(\mathbb{S}^{n-1}, d^{n-1}\omega)}\, r^{n-1} d r  \no  \\
& \quad = (n-4) \int_{0}^{\infty} r^{-4}(g, g)_{L^2(\mathbb{S}^{n-1}, d^{n-1}\omega)}\, r^{n-1} d r+\int_{0}^{\infty} r^{-2}(g^{\prime}, g^{\prime})_{L^2(\mathbb{S}^{n-1}, d^{n-1}\omega)}\, r^{n-1} d r . \no
\end{align}
\end{proof}

We now turn to applying Lemmas \ref{l3.3} and \ref{l3.4} to show some useful identities and estimates needed in the proof of Theorem \ref{t3.1}. 

\begin{lemma}\lb{l3.5}
Let $\g\in\bbR$ and $f\in C_0^\infty(\bbR^n\backslash\{0\})$ be real-valued, $n\in\bbN$,  $n\geq2$. Then,
\begin{align}
& - \int_0^\infty\int_{{\bbS^{n-1}}}r^{\frac{\g}{2}}\big\{r^{-(n-1)}(\partial/\partial r)    \no \\
& \hspace*{2.5cm}  \times \big[r^{n-1}(\partial/\partial r)\big(r^{\frac{\g}{2}-2}(-\Delta_{\bbS^{n-1}} f)(r,\t)\big)\big]\big\}f(r,\t)\, d^{n-1} \omega(\theta) \, r^{n-1}dr   \no \\
&\quad=\big(2\g-4^{-1}\g^2-2^{-1}n\g+n-4\big)     \no \\ 
& \qquad \quad \times \int_0^\infty\int_{{\bbS^{n-1}}}r^{\g-4}\big|\big((-\Delta_{\bbS^{n-1}})^{1/2}f\big)(r,\t)\big|^2\, d^{n-1} \omega(\theta) \, r^{n-1}dr  \no  \\
&\qquad +\int_0^\infty r^{\g-2}\big\|(d/dr)\big(\big((-\Delta_{\bbS^{n-1}})^{1/2}f\big)(r,\dott)\big)\big\|^2_{L^2({\bbS^{n-1}};d^{n-1}\omega)}r^{n-1}dr, \lb{3.23}
\end{align}
and hence,
\begin{align}
& - \int_0^\infty\int_{{\bbS^{n-1}}}r^{\frac{\g}{2}}\big\{r^{-(n-1)}(\partial/\partial r)\big[r^{n-1}(\partial/\partial r)  \no \\  & \hspace*{2.5cm} \times \big(r^{\frac{\g}{2}-2}(-\Delta_{\bbS^{n-1}} f)(r,\t)\big)\big]\big\}f(r,\t)\, d^{n-1} \omega(\theta) \, r^{n-1}dr   \no \\
& \quad \geq 4^{-1}n(n-4)\int_0^\infty\int_{{\bbS^{n-1}}}r^{\g-4}\big|\big((-\Delta_{\bbS^{n-1}})^{1/2}f\big)(r,\t)\big|^2\, d^{n-1} \omega(\theta) \, r^{n-1}dr.  \lb{3.24}
\end{align}
\end{lemma}
\begin{proof}
Since the resolvent of $-\Delta_{\mathbb{S}^{n-1}}$ is an integral operator with positive integral kernel, if $f\in C_0^\infty\big(\mathbb{R}^n\backslash\{0\}\big)$ is real-valued, then by \cite[Sect.~3.1]{Ha06}, or,  \cite[eq.~(V.3.45)]{Ka80a}, this implies that  $(-\Delta_{\mathbb{S}^{n-1}})^{1/2} f$ is real-valued. So we can apply Lemma \ref{l3.3} and Lemma \ref{l3.4} with  $g\in C_0^\infty((0,\infty),  L^2(\mathbb{S}^{n-1}, d^{n-1}\omega))$ being  the function
\begin{equation}
r\mapsto r^{\gamma/2}\big((-\Delta_{\mathbb{S}^{n-1}})^{1/2}f\big)(r,\cdot).
\end{equation} 
One verifies that 
\begin{align}
&- \int_0^\infty\int_{{\bbS^{n-1}}}r^{\frac{\g}{2}}\big\{r^{-(n-1)}(\partial/\partial r)     \no \\
& \hspace*{2.5cm} \times \big[r^{n-1}(\partial/\partial r)\big(r^{\frac{\g}{2}-2}(-\Delta_{\bbS^{n-1}} f)(r,\t)\big)\big]\big\}f(r,\t)\, d^{n-1} \omega(\theta) \, r^{n-1}dr   \no \\
&\quad = - \int_0^\infty\int_{{\bbS^{n-1}}}(-\Delta_{\bbS^{n-1}})^{1/2}\big\{r^{-(n-1)}(\partial/\partial r)\big[r^{n-1}(\partial/\partial r)\big(r^{\frac{\g}{2}-2}f(r,\t)\big)\big]\big\}  \no  \\
&\hspace*{2.6cm} \times(-\Delta_{\bbS^{n-1}})^{1/2}(r^{\g/2}f)(r,\t)\, d^{n-1} \omega(\theta) \, r^{n-1}dr  \no  \\
& \quad = - \int_0^\infty\int_{{\bbS^{n-1}}} r^{-(n-1)}(d/d r)\big[r^{n-1}(d/d r)\big(r^{\frac{\g}{2}-2}\big((-\Delta_{\bbS^{n-1}})^{1/2}f\big)(r,\t)\big)\big]  \no  \\
&\hspace*{2.6cm} \times r^{\g/2}\big((-\Delta_{\bbS^{n-1}})^{1/2}f\big)(r,\t)\, d^{n-1} \omega(\theta) 
\, r^{n-1}dr  \no  \\
& \quad = - \int_0^\infty \big(r^{-(n-1)}(d/dr)\big(r^{n-1}(d/dr)\big(r^{-2}g\big)\big),g\big)_{L^2({\bbS^{n-1}};d^{n-1}\omega)}\, r^{n-1}dr \no  \\
& \quad =\int_0^\infty r^{-2}(g',g')_{L^2({\bbS^{n-1}};d^{n-1}\omega)}\, r^{n-1}dr+(n-4)\int_0^\infty r^{-4}(g,g)_{L^2({\bbS^{n-1}};d^{n-1}\omega)}\, r^{n-1}dr \no  \\
& \quad =\big(4^{-1}\g^2+n-4\big)\int_0^\infty \int_{\bbS^{n-1}} r^{\g-4}\big|\big((-\Delta_{\bbS^{n-1}})^{1/2}f\big)(r,\t)\big|^2\, d^{n-1}\omega(\theta) \, r^{n-1}dr  \no  \\
&\qquad +\int_0^\infty r^{\g-2}\norm{\frac{d}{dr}\big((-\Delta_{\bbS^{n-1}})^{1/2}f(r,\dott)\big)}^2_{L^2({\bbS^{n-1}};d^{n-1}\omega)}\, r^{n-1}dr \no  \\
&\qquad +\g\int_0^\infty r^{\g-3}\big((-\Delta_{\bbS^{n-1}})^{1/2}f,(d/dr)\big((-\Delta_{\bbS^{n-1}})^{1/2}f\big)\big)_{L^2({\bbS^{n-1}};d^{n-1}\omega)}\, r^{n-1}dr \no  \\
&\quad =\big(4^{-1}\g^2+n-4\big)\int_0^\infty \int_{\bbS^{n-1}} r^{\g-4}\big|\big((-\Delta_{\bbS^{n-1}})^{1/2}f\big)(r,\t)\big|^2\, d^{n-1}\omega(\theta) \, r^{n-1}dr  \no  \\
&\qquad +\int_0^\infty r^{\g-2}\norm{\frac{d}{dr}\big(\big((-\Delta_{\bbS^{n-1}})^{1/2}f\big)(r,\dott)\big)}^2_{L^2({\bbS^{n-1}};d^{n-1}\omega)}\, r^{n-1}dr \no  \\
&\qquad +\frac{\g}{2}\int_0^\infty r^{\g+n-4}(d/dr)\big[\big((-\Delta_{\bbS^{n-1}})^{1/2}f,(-\Delta_{\bbS^{n-1}})^{1/2}f\big)_{L^2({\bbS^{n-1}};d^{n-1}\omega)}\big]\, dr \no  \\
&\quad =\big(2\g-4^{-1}\g^2-2^{-1}n\g+n-4\big)     \no \\ 
& \qquad \quad \times \int_0^\infty \int_{\bbS^{n-1}} r^{\g-4}\big|\big((-\Delta_{\bbS^{n-1}})^{1/2}f\big)(r,\t)\big|^2\, d^{n-1}\omega(\theta) \, r^{n-1}dr  \no  \\
&\qquad +\int_0^\infty r^{\g-2}\norm{\frac{d}{dr}\big(\big((-\Delta_{\bbS^{n-1}})^{1/2}f\big)(r,\dott)\big)}^2_{L^2({\bbS^{n-1}};d^{n-1}\omega)}\, r^{n-1}dr,
\end{align}
where the last step follows from integration by parts and collecting similar terms. This proves \eqref{3.23}. Finally, by \cite[Theorem 4.2]{GLMP22},
\begin{align}
& \int_0^\infty r^{\g-2} \norm{\frac{d}{dr}\big(\big((-\Delta_{\bbS^{n-1}})^{1/2}f\big)(r,\dott)\big)}^2_{L^2({\bbS^{n-1}};d^{n-1}\omega)}\, r^{n-1}dr \lb{3.27} \\
& \quad \geq 4^{-1}(\gamma+n-4)^{2}\int_0^\infty \int_{\bbS^{n-1}} r^{\g-4}\big|\big((-\Delta_{\bbS^{n-1}})^{1/2}f\big)(r,\t)\big|^2\, d^{n-1}\omega(\theta) \, r^{n-1}dr.  \no 
\end{align}
Thus, \eqref{3.24} follows from \eqref{3.23} and \eqref{3.27}.
\end{proof}

The strategy of proof used for Lemma \ref{l3.5} is employed in a similar fashion in proving the following lemmas. As such, we will mainly dwell on the important differences in the proofs below.

\begin{lemma}\lb{l3.6}
Let $\g\in\bbR$ and $f\in C_0^\infty(\bbR^n\backslash\{0\})$ be real-valued, $n\in\bbN$, $n\geq2$. Then,
\begin{align}
\begin{split}
& \int_0^\infty\int_{\bbS^{n-1}} r^{\frac{\g}{2}-1}\big[(\partial/\partial r)\big(r^{\frac{\g}{2}-2}(-\Delta_{\bbS^{n-1}} f)(r,\t)\big)\big]f(r,\t)\,  d^{n-1}\omega(\theta) \, r^{n-1}dr\\
&\quad =-\frac{n}{2}\int_0^\infty\int_{\bbS^{n-1}} r^{\g-4}\big|\big((-\Delta_{\bbS^{n-1}})^{1/2}f\big)(r,\t)\big|^2\, d^{n-1}\omega(\theta) \, r^{n-1}dr.
\end{split}
\end{align}
\end{lemma}
\begin{proof}
The result follows similarly to the proof of Lemma \ref{l3.5} by applying Lemma \ref{l3.3}.
\end{proof}

\begin{lemma}\lb{l3.7}
Let $\g\in \bbR$ and $f\in C_0^\infty(\bbR^n\backslash\{0\})$ be real-valued, $n\in\bbN$,  $n\geq 2$. Then, 
\begin{align}
& \int_0^\infty\int_{\bbS^{n-1}} r^{\frac{\g}{2}-2}\big[(-\Delta_{\bbS^{n-1}})\big(-r^{\frac{\g}{2}-(n-1)}(\partial/\partial r)
\no \\
& \hspace*{2.5cm} \times \big(r^{n-1}(\partial f/\partial r)(r,\t)\big)\big)\big]f(r,\t)\,  d^{n-1}\omega(\theta) \, r^{n-1}dr  \no  \\
&\quad =-2^{-1}(\g-2)(\g+n-4)\int_0^\infty\int_{\bbS^{n-1}} r^{\g-4}\big|\big((-\Delta_{\bbS^{n-1}})^{1/2}f\big)(r,\t)\big|^2\, d^{n-1}\omega(\theta) \, r^{n-1}dr  \no  \\
&\qquad +\int_0^\infty r^{\g-2}\norm{\frac{d}{dr}\big(\big((-\Delta_{\bbS^{n-1}})^{1/2}f\big)(r,\dott)\big)}^2_{L^2({\bbS^{n-1}};d^{n-1}\omega)}\, r^{n-1}dr,  \lb{3.29}
\end{align}
and hence,
\begin{align}
&\int_0^\infty\int_{\bbS^{n-1}} r^{\frac{\g}{2}-2}\big[(-\Delta_{\bbS^{n-1}})\big(-r^{\frac{\g}{2}-(n-1)}(\partial/\partial r) 
\no \\
& \hspace*{2.5cm} \times \big(r^{n-1}(\partial f/\partial r)(r,\t)\big)\big)\big]f(r,\t)\,  d^{n-1}\omega(\theta) \, r^{n-1}dr  \lb{3.30} \\
&\quad \geq4^{-1}(n-\g)(\g+n-4)\int_0^\infty\int_{\bbS^{n-1}} r^{\g-4}\big|\big((-\Delta_{\bbS^{n-1}})^{1/2}f\big)(r,\t)\big|^2\, d^{n-1}\omega(\theta) \, r^{n-1}dr.   \no
\end{align}
\end{lemma}
\begin{proof}
A similar calculation as before, applying Lemma \ref{l3.3}, yields
\begin{align}
&\int_0^\infty\int_{\bbS^{n-1}} r^{\frac{\g}{2}-2}\big[(-\Delta_{\bbS^{n-1}})\big(-r^{\frac{\g}{2}-(n-1)}  
\no \\ 
& \hspace*{2.5cm} \times (\partial/\partial r)\big(r^{n-1}(\partial f/\partial r)(r,\t)\big)\big)\big]f(r,\t)\,  d^{n-1}\omega(\theta) \, r^{n-1}dr  \no  \\
&\quad =-\int_0^\infty\int_{\bbS^{n-1}} r^{\g-n-1}\big[(n-1)r^{n-2}\big((-\Delta_{\bbS^{n-1}})^{1/2}(\partial f/\partial r)(r,\t)\big)    \no \\
&\qquad + r^{n-1}(d/dr) \big((-\Delta_{\bbS^{n-1}})^{1/2}(\partial f/\partial r)(r,\t)\big)\big]\big((-\Delta_{\bbS^{n-1}})^{1/2}f\big)(r,\t)\,  d^{n-1}\omega(\theta) \, r^{n-1}dr  \no  \\  
&\quad =(1-n)\int_0^\infty\int_{\bbS^{n-1}} r^{\g-3}(d/dr)\big[\big((-\Delta_{\bbS^{n-1}})^{1/2}f\big)(r,\t)\big]    \no \\
& \hspace*{3.5cm} \times \big((-\Delta_{\bbS^{n-1}})^{1/2}f\big)(r,\t)\,  d^{n-1}\omega(\theta) \, r^{n-1}dr  \lb{3.31} \\
&\qquad -\int_0^\infty\int_{\bbS^{n-1}} r^{\g-2}(d^2/dr^2)\big[\big((-\Delta_{\bbS^{n-1}})^{1/2}f\big)(r,\t)\big] \no \\
& \hspace*{2.6cm} \times \big((-\Delta_{\bbS^{n-1}})^{1/2}f\big)(r,\t)\,  d^{n-1}\omega(\theta) \, r^{n-1}dr.  \no 
\end{align}
When considering the first integral on the right-hand side of \eqref{3.31}, integration by parts yields
\begin{align}
\begin{split} 
& \int_0^\infty\int_{\bbS^{n-1}} r^{\g-3}(d/dr)\big[\big((-\Delta_{\bbS^{n-1}})^{1/2}f\big)(r,\t)\big]\big((-\Delta_{\bbS^{n-1}})^{1/2}f\big)(r,\t)\,  d^{n-1}\omega(\theta) \, r^{n-1}dr \lb{3.32} \\
& \quad = -2^{-1}(\g+n-4)  \int_0^\infty\int_{\bbS^{n-1}} r^{\g-4}\big|\big((-\Delta_{\bbS^{n-1}})^{1/2}f\big)(r,\t)\big|^2\, d^{n-1}\omega(\theta) \, r^{n-1}dr.  
\end{split} 
\end{align}
The second integral on the right-hand side of \eqref{3.31} differs slightly from previous considerations due to the second derivative, however, it follows an analogous calculation by first noting that given $g\in C_0^\infty((0,\infty),L^2({\bbS^{n-1}};d^{n-1}\omega))$ with $g(r)$ real-valued on $\bbS^{n-1}$ for all $r>0$, one has
\begin{equation}
(g'',g)_{L^2({\bbS^{n-1}};d^{n-1}\omega)}=\frac{d}{dr}(g,g')_{L^2({\bbS^{n-1}};d^{n-1}\omega)}-(g',g')_{L^2({\bbS^{n-1}};d^{n-1}\omega)}. \lb{3.33}
\end{equation}
Next, letting $g(r)=\big((-\Delta_{\bbS^{n-1}})^{1/2}f\big)(r,\dott),\ r>0$, applying \eqref{3.33}, and two integrations by parts, one gets 
\begin{align}
& \int_0^\infty \int_{\bbS^{n-1}} r^{\g-2}(d^2/dr^2)\big[\big((-\Delta_{\bbS^{n-1}})^{1/2}f\big)(r,\t)\big]\big((-\Delta_{\bbS^{n-1}})^{1/2}f\big)(r,\t)\,  d^{n-1}\omega(\theta) \, r^{n-1}dr  \no \\
& \quad =2^{-1}(\g+n-3)(\g+n-4)\int_0^\infty\int_{\bbS^{n-1}} r^{\g-4}\big|\big((-\Delta_{\bbS^{n-1}})^{1/2}f\big)(r,\t)\big|^2\, d^{n-1}\omega(\theta) \, r^{n-1}dr  \no  \\
&\qquad -\int_0^\infty r^{\g-2}\norm{\frac{d}{dr}\big(\big((-\Delta_{\bbS^{n-1}})^{1/2}f\big)(r,\dott)\big)}^2_{L^2({\bbS^{n-1}};d^{n-1}\omega)}\, r^{n-1}dr.  \lb{3.34}
\end{align}
Combining \eqref{3.31}, \eqref{3.32}, and \eqref{3.34} proves \eqref{3.29}.

Finally, \eqref{3.30} follows from an application of \cite[Theorem 4.2]{GLMP22} as before.
\end{proof}

\begin{lemma}\lb{l3.8}
Let $\g\in \bbR$ and $f\in C_0^\infty(\bbR^n\backslash\{0\})$ be real-valued, $n\in\bbN$, $n\geq2$. Then,
\begin{align}
& \int_0^\infty\int_{\bbS^{n-1}} r^{\frac{\g}{2}-2}\big[-\Delta_{\bbS^{n-1}}\big(r^{\frac{\g}{2}-1}(\partial f/\partial r)(r,\t)\big)\big]f(r,\t)\, d^{n-1}\omega(\theta) \, r^{n-1}dr \\
& \quad =-2^{-1}(\g+n-4)\int_0^\infty\int_{\bbS^{n-1}} r^{\g-4}\big|\big((-\Delta_{\bbS^{n-1}})^{1/2}f\big)(r,\t)\big|^2\, d^{n-1}\omega(\theta) \, r^{n-1}dr.  \no  
\end{align}
\end{lemma}
\begin{proof}
The result follows similarly to the proof of Lemma \ref{l3.5} by applying Lemma \ref{l3.3}.
\end{proof}

We are now in a position to provide a proof of Theorem \ref{t3.1}.

\begin{proof}[Proof of Theorem \ref{3.1}]
The proof is analogous to that of Theorem \ref{t2.1} by first noting that
\begin{align}
\begin{split} 
\big(T_{\a,\b,\g,\tau}^+T_{\a,\b,\g,\tau}f\big)&=\big(T_{\a,\b,\g}^+T_{\a,\b,\g}f\big)+\tau \big[T_{\a,\b,\g}^+\big(r^{\frac{\g}{2}-2}(-\Delta_{\bbS^{n-1}} f)\big)\big]  \\
&\quad\; +\tau r^{\frac{\g}{2}-2}\big(-\Delta_{\bbS^{n-1}} T_{\a,\b,\g}f\big)+\tau^2 r^{\g-4}\big((-\Delta_{\bbS^{n-1}})^2f\big).    
\end{split} 
\end{align}
The remaining calculations are analogous to the previous results by applying Theorem \ref{t2.1} and Lemmas \ref{l3.5}--\ref{l3.8}, expanding each integral, and gathering similar terms.
\end{proof}

We now turn to some implications of Theorem \ref{t3.1} that are extensions of the results presented in Section \ref{s2}, beginning with an extension of Corollary \ref{c2.3}.

\begin{corollary}\lb{c3.9}
Let $\g,\tau\in\bbR$ and $f \in C^{\infty}_0(\bbR^n \backslash \{0\})$, $n\in\bbN$, $n\geq2$. Then,
\begin{align} 
& \int_{\bbR^n} |x|^\g|(\Delta f)(x)|^2 \, d^n x 
\geq \bigg[\frac{(n-2)^2}{4}-\frac{(\g-2)^2}{4}\bigg]^2 \int_{\bbR^n} |x|^{\g-4} |f(x)|^2 \, d^n x    \no \\
&\quad+2^{-1}\big[\tau(\g+n-4)^2+(n-2)^2-(\g-2)^2 \big] \no \\
&\qquad \times \int_0^\infty\int_{{\bbS^{n-1}}}r^{\g-4}\big|\big((-\Delta_{\bbS^{n-1}})^{1/2}f\big)(r,\t)\big|^2\,  d^{n-1} \omega(\theta) \, r^{n-1}dr       \lb{3.38}  \\
&\quad -2\tau\int_0^\infty\int_{{\bbS^{n-1}}}r^{\g-2}\bigg|\bigg((-\Delta_{\bbS^{n-1}})^{1/2}\frac{\partial f}{\partial r}\bigg)(r,\t)\bigg|^2\,  d^{n-1} \omega(\theta) \, r^{n-1}dr   \no  \\
&\quad -\tau(\tau+2)\int_0^\infty\int_{{\bbS^{n-1}}}r^{\g-4}|(\Delta_{\bbS^{n-1}} f)(r,\t)|^2\, d^{n-1} \omega(\theta) \, r^{n-1}dr.  \no 
\end{align} 
\end{corollary}
\begin{proof}
Without loss of generality (cf.\ Remark \ref{r1.1}) we can assume that $f \in C^{\infty}_0(\bbR^n \backslash \{0\})$ is real-valued. As in the proof of Corollary \ref{c2.3}, one chooses in \eqref{3.1} 
\begin{equation}\lb{3.39}
\alpha=\alpha_{\pm} = 2-\g \pm \big[\big(\g^2-4\g+n^2-4n+8\big)/2\big]^{1/2}\quad \text{and} \quad \b=\a(n-\a-\g)/2.
\end{equation}
Direct computation with these choices yields
\begin{equation}
-\tau[(\g+n-4)(2-\a-\g)+2\b]=2^{-1}\tau(\g+n-4)^2,\quad \g,\tau\in\bbR,\ n\geq2.
\end{equation}
Next, letting $\nabla_{\bbS^{n-1}}$ denote the gradient operator on ${\bbS^{n-1}}$, one notes that
\begin{equation}\lb{3.41}
|(\nabla f)(x)|^2=|(\partial f/\partial r)(r,\theta)|^2+r^{-2}|(\nabla_{\bbS^{n-1}} f(r,\dott))(\theta)|^2.
\end{equation}
Furthermore, with $\a$ and $\b$ chosen as in \eqref{3.39}, one confirms that
\begin{equation}\lb{3.42}
\a(\g+n-4)-2\b=\a(\a-4+2\g)=2^{-1}\big[(n-2)^2-(\g-2)^2\big],
\end{equation}
so that by \eqref{3.39} and \eqref{3.42}, the first two integrals on the right-hand side of \eqref{3.1} can be combined to give
\begin{align}\lb{3.43}
&2^{-1}\big[(n-2)^2-(\g-2)^2 \big] \int_0^\infty\int_{{\bbS^{n-1}}}r^{\g-4}|(\nabla_{\bbS^{n-1}} f)(r,\t)|^2\,  d^{n-1} \omega(\theta) \, r^{n-1}dr     \\
&=2^{-1}\big[(n-2)^2-(\g-2)^2 \big] \int_0^\infty\int_{{\bbS^{n-1}}}r^{\g-4}\big|\big((-\Delta_{\bbS^{n-1}})^{1/2}f\big)(r,\t)\big|^2\,  d^{n-1} \omega(\theta) \, r^{n-1}dr,    \no 
\end{align}
where equality comes from integration by parts and utilizing the $L^2(S^{n-1};d^{n-1}\omega)$ inner product similar to previous calculations, completing the proof.
\end{proof}

\begin{remark}\lb{r3.10}
Setting $\tau=0$ in \eqref{3.38} yields an extension of Corollary \ref{c2.3} in the sense of the additional spherical term given in \eqref{3.43}, which, under the assumptions of Corollary \ref{c2.3}, will have a positive coefficient.
\hfill$\diamond$
\end{remark}

\begin{theorem}[{\cite[Theorem 3.1]{CM12}}]\lb{t3.11}
Let $\g\in\bbR$ and $f \in C^{\infty}_0(\bbR^n \backslash \{0\})$, $n\in\bbN$, $n\geq2$. Then,
\begin{equation}\lb{3.44}
\int_{\bbR^n} |x|^\g|(\Delta f)(x)|^2 \, d^n x 
 \geq C_{n,\g} \int_{\bbR^n} |x|^{\g-4} |f(x)|^2 \, d^n x,
\end{equation}
where
\begin{equation}\lb{3.45}
C_{n,\g}=\min_{j\in\bbN_0}\Big\{\big[j(j+n-2) + 4^{-1}(n-2)^2 - 4^{-1}(\g-2)^2\big]^2\Big\}.
\end{equation}
\end{theorem}
\begin{proof}
First, applying \cite[Theorem 4.2]{GLMP22} as in \eqref{3.27}, one recalls that
\begin{align}
& \int_0^\infty\int_{{\bbS^{n-1}}}r^{\g-2} \bigg|\bigg((-\Delta_{\bbS^{n-1}})^{1/2}\frac{\partial f}{\partial r}\bigg)(r,\t)\bigg|^2\,  d^{n-1} \omega(\theta) \, r^{n-1}dr \no  \\
& \quad =\int_0^\infty r^{\g-2}\norm{\frac{d}{dr}\big(\big((-\Delta_{\bbS^{n-1}})^{1/2}f\big)(r,\dott)\big)}^2_{L^2({\bbS^{n-1}};d^{n-1}\omega)}\, r^{n-1}dr \lb{3.46} \\
& \qquad \geq 4^{-1}(\gamma+n-4)^{2}\int_0^\infty \int_{\bbS^{n-1}} r^{\g-4}\big|\big((-\Delta_{\bbS^{n-1}})^{1/2}f\big)(r,\t)\big|^2\, d^{n-1}\omega(\theta) \, r^{n-1}dr.  \no 
\end{align}
Therefore, for all $\tau \in (-\infty,0)$, \eqref{3.38} and \eqref{3.46} imply that
\begin{align} 
& \int_{\bbR^n} |x|^\g|(\Delta f)(x)|^2 \, d^n x 
\geq \bigg[\frac{(n-2)^2}{4}-\frac{(\g-2)^2}{4}\bigg]^2 \int_{\bbR^n} |x|^{\g-4} |f(x)|^2 \, d^n x    \no \\
&\quad+2^{-1}\big[(n-2)^2-(\g-2)^2 \big] \no \\
&\qquad \times \int_0^\infty\int_{{\bbS^{n-1}}}r^{\g-4}\big|\big((-\Delta_{\bbS^{n-1}})^{1/2}f\big)(r,\t)\big|^2\,  d^{n-1} \omega(\theta) \, r^{n-1}dr     \lb{3.47} \\
&\quad -\tau(\tau+2)\int_0^\infty\int_{{\bbS^{n-1}}}r^{\g-4}|(-\Delta_{\bbS^{n-1}} f)(r,\t)|^2\, d^{n-1} \omega(\theta) \, r^{n-1}dr.  \no 
\end{align}
Finally, choosing $\tau=-1$ in \eqref{3.47} and applying the spectral theorem to $-\Delta_{\bbS^{n-1}}$ yields
\begin{align}
& \int_{\bbR^n} |x|^\g|(\Delta f)(x)|^2 \, d^n x 
\geq \bigg[\frac{(n-2)^2}{4}-\frac{(\g-2)^2}{4}\bigg]^2 \int_{\bbR^n} |x|^{\g-4} |f(x)|^2 \, d^n x   \no \\
&\qquad+2^{-1}\big[(n-2)^2-(\g-2)^2 \big] \int_0^\infty\int_{{\bbS^{n-1}}}r^{\g-4}\big|\big((-\Delta_{\bbS^{n-1}})^{1/2}f\big)(r,\t)\big|^2\,  d^{n-1} \omega(\theta) \, r^{n-1}dr  \no  \\
&\qquad +\int_0^\infty\int_{{\bbS^{n-1}}}r^{\g-4}|(-\Delta_{\bbS^{n-1}} f)(r,\t)|^2\, d^{n-1} \omega(\theta) \, r^{n-1}dr  \no  \\
&\quad = \int_0^\infty r^{\g-4} \bigg\{\sum_{j\in\bbN_0}\sum_{\ell=0}^{m(\lambda_j)}\bigg(\bigg[\frac{(n-2)^2}{4}-\frac{(\g-2)^2}{4}\bigg]^2  \no  \\
&\hspace*{2.35cm}+2\bigg[\frac{(n-2)^2}{4}-\frac{(\g-2)^2}{4}\bigg]\lambda_j+\lambda_j^2\bigg)\big|F_{f,j,\ell}(r)\big|^2 \, r^{n-1}dr  \no \\
&\quad \geq \inf_{j\in\bbN_0}\Big\{\big[4^{-1}(n-2)^2-4^{-1}(\g-2)^2+\lambda_j\big]^2\Big\} 
 \int_0^\infty r^{\g-4}  \sum_{j\in\bbN_0}\sum_{\ell=0}^{m(\lambda_j)}\big|F_{f,j,\ell}(r)\big|^2 \, r^{n-1}dr \no  \\
&\quad =C_{n,\g} \int_{\bbR^n} |x|^{\g-4} |f(x)|^2 \, d^n x,   \lb{3.48}
\end{align}
with $C_{n,\g}$ given in \eqref{3.45}.
\end{proof}

\begin{remark}\lb{r3.12}
It has been shown in \cite[p.~148--149, Theorem~3.1]{CM12} that the constant $C_{n,\g}$ in \eqref{3.44}, \eqref{3.45} is sharp for all $\gamma\in \bbR$ and $n\in\bbN,\ n\geq2$. Remark \ref{r3.20a} provides another proof. For $\gamma>4-n$, the sharpness of $C_{n,\gamma}$ was shown in \cite[Theorem 6.5.2]{GM13} also. For all $\gamma\in \bbR$ and $n\in\bbN,\ n\geq2$, there are now two proofs of this optimal Rellich inequality: the original proof in \cite{CM12} and the proof by factorization given in Theorem \ref{t3.11} combined with Remark \ref{r3.20a}. (We also note that \cite[Theorem 3.5]{GMP24} derives inequality \eqref{3.44}, but sharpness of $C_{n,\gamma}$ was not demonstrated in \cite{GMP24}.) 
\hfill$\diamond$
\end{remark}

\begin{remark} \lb{r3.13}
As discussed in \cite{CM12}, whenever, $4^{-1} \big[(\gamma - 2)^2 - (n-2)^2\big]$ equals one of the eigenvalues of $- \Delta_{\bbS^{n-1}}$ (i.e., one of the numbers $j(j+n-2)$, $j \in \bbN_0$), then $C_{n,\gamma}$ vanishes, rendering inequality \eqref{3.44} trivial. In this context we recall the following result from \cite[Theorem~1.3]{GMP24}: 
\begin{align}
\begin{split} 
&\int_{B_n(0;R)} |x|^{\gamma} |(\Delta f)(x)|^{2} \, d^{n}x \geq C_{n,\gamma} \int_{B_n(0;R)}  
|x|^{\gamma-4} | f(x)|^{2} \, d^{n}x    \\
&\quad + \big\{\big[ (n-\gamma)^{2} + (n+\gamma-4)^{2} \big]\big/16\big\}     \lb{3.48a} \\
&\qquad \times \int_{B_n(0;R)} |x|^{\gamma-4}  \Bigg(\sum_{k=1}^{N} \prod_{p=1}^{k} [\ln_{p}(\eta/|x|)]^{-2} \Bigg)|f(x)|^{2} \, d^{n}x,  \\
& \, R \in (0,\infty), \; \gamma \in \bbR, \; N, n \in \bbN, \, n \geq 2, \; \eta \in [e_{N}R,\infty), \; 
f \in C_{0}^{\infty}(B_n(0;R)\bs\{0\}), 
\end{split} 
\end{align}
which yields an appropriate logarithmic refinement even if $C_{n,\gamma}$ vanishes. Here $B_n(0;R)$ denotes the ball in $\bbR^n$, $n \in \bbN$, $n \geq 2$, centered at the origin $0$ of radius $R \in (0,\infty)$, the iterated logarithms $\ln_k( \, \cdot \, )$, $k \in \bbN$, are given by
\begin{equation} 
\ln_1(\, \cdot \,) = \ln( \, \cdot \, ), \quad \ln_{k+1}( \, \cdot \,) = \ln \big( \ln_k(\, \cdot \,) \big), \quad k \in \bbN, 
\end{equation} 
and the iterated exponentials $e_j$, $j \in \bbN_0$, are introduced via 
\begin{equation}
e_0 = 0, \quad e_{j+1} = e^{e_j}, \quad j \in \bbN_{0}.
\end{equation} 
\hfill $\diamond$
\end{remark}

Motivated by the estimation given in \eqref{3.46}, one immediately arrives at the following inequality:

\begin{lemma}\lb{l3.13}
Let $\a,\b,\g\in\bbR$, $\tau \in (-\infty,0)$, and $f\in C_0^\infty(\bbR^n\backslash\{0\})$, $n\in\bbN$,  $n\geq2$. Then,
\begin{align} 
\begin{split} 
&\int_{\bbR^n}|x|^\g |(\Delta f)(x)|^2\, d^nx \geq [\a(\g+n-4)-2\b]\int_{\bbR^n}|x|^{\g-2}|(\nabla f)(x)|^2\, d^nx  \\
&\quad-\a(\a-4+2\g)\int_{\bbR^n}|x|^{\g-4}|x\cdot(\nabla f)(x)|^2\, d^nx   \\
&\quad+\b[(n-4)(\a-2)-\b+\g(n+\a+\g-6)]\int_{\bbR^n}|x|^{\g-4}|f(x)|^2 d^nx     \lb{3.49} \\
&\quad-\tau\big[2\b+(\g+n-4)\big(2^{-1}n-2^{-1}\g-\a\big)\big]   \\
&\qquad \times \int_0^\infty\int_{{\bbS^{n-1}}}r^{\g-4}\big|\big((-\Delta_{\bbS^{n-1}})^{1/2}f\big)(r,\t)\big|^2\,  d^{n-1} \omega(\theta) \, r^{n-1}dr   \\
&\quad -\tau(\tau+2)\int_0^\infty\int_{{\bbS^{n-1}}}r^{\g-4}|(\Delta_{\bbS^{n-1}} f)(r,\t)|^2\, d^{n-1} \omega(\theta) \, r^{n-1}dr.  
\end{split} 
\end{align}
\end{lemma}
\begin{proof}
Without loss of generality (cf.\ Remark \ref{r1.1}) we can assume that $f \in C^{\infty}_0(\bbR^n \backslash \{0\})$ is real-valued. Applying \eqref{3.46} to \eqref{3.1} and noting that
\begin{equation}
-\tau[(\g+n-4)(2-\a-\g)+2\b]-2\tau\big[4^{-1}(\g+n-4)^2\big]=-\tau\big[2\b+(\g+n-4)\big(2^{-1}n-2^{-1}\g-\a\big)\big]
\end{equation}
yields \eqref{3.49}.
\end{proof}

Choosing $\tau\in(-2,0)$ as previously mentioned in Remark \ref{r3.2} allows one to estimate the right-hand side of \eqref{3.49} further as follows.

\begin{lemma}\lb{l3.14}
Let $\a,\b,\g\in\bbR$, $\tau\in(-2,0)$, and $f\in C_0^\infty(\bbR^n\backslash\{0\})$, $n\in\bbN$,  $n\geq2$. Then,
\begin{align}
& \int_{\bbR^n}|x|^\g |(\Delta f)(x)|^2\, d^nx \geq [\a(\g+n-4)-2\b]\int_{\bbR^n}|x|^{\g-2}|(\nabla f)(x)|^2\, d^nx \no \\
&\quad-\a(\a-4+2\g)\int_{\bbR^n}|x|^{\g-4}|x\cdot(\nabla f)(x)|^2\, d^nx    \no \\
&\quad+\b[(n-4)(\a-2)-\b+\g(n+\a+\g-6)]\int_{\bbR^n}|x|^{\g-4}|f(x)|^2 d^nx    \lb{3.51} \\
&\quad-\tau\big[2\b+(\g+n-4)\big(2^{-1}n-2^{-1}\g-\a\big)+(\tau+2)(n-1)\big]  \no \\
&\qquad \times \int_0^\infty\int_{{\bbS^{n-1}}}r^{\g-4}\big|\big((-\Delta_{\bbS^{n-1}})^{1/2}f\big)(r,\t)\big|^2\,  d^{n-1} \omega(\theta) \, r^{n-1}dr.  \no 
\end{align}
\end{lemma}
\begin{proof}
Without loss of generality (cf.\ Remark \ref{r1.1}) we can assume that $f \in C^{\infty}_0(\bbR^n \backslash \{0\})$ is real-valued. One notes that by the spectral theorem for $-\Delta_{\bbS^{n-1}}$ one has
\begin{align} 
& \int_0^\infty\int_{{\bbS^{n-1}}}r^{\g-4} |(\Delta_{\bbS^{n-1}} f)(r,\t)|^2\, d^{n-1} \omega(\theta) \, r^{n-1}dr    \no \\
&\quad =\sum_{j\in\bbN_0}\sum_{\ell=1}^{m(\lambda_j)}\lambda_j^2\int_0^\infty r^{\g-4} 
|F_{f,j,\ell}(r)|^2\, r^{n-1}dr  \no   \\
&\quad \geq (n-1)\sum_{j\in\bbN_0}\sum_{\ell=1}^{m(\lambda_j)}\lambda_j\int_0^\infty r^{\g-4}
|F_{f,j,\ell}(r)|^2\, r^{n-1}dr  \no   \\
&\quad =(n-1)\int_0^\infty\int_{{\bbS^{n-1}}}r^{\g-4}\big|\big((-\Delta_{\bbS^{n-1}})^{1/2}f\big)(r,\t)\big|^2\,  d^{n-1} \omega(\theta) \, r^{n-1}dr,    \lb{3.52}
\end{align}
from which substituting \eqref{3.52} into \eqref{3.49} for $\tau\in(-2,0)$ yields \eqref{3.51}.
\end{proof}

We end this section with a further extension of Schmincke's inequality (see Theorem \ref{t2.7} and \cite{Sc72}) in this setting:
\begin{theorem}\lb{t3.19}
Let $\g\in\bbR$ and $f\in C_0^\infty(\bbR^n\backslash\{0\})$, $n\in\bbN$, $n\geq2$. Consider the inequality
\begin{align}
\begin{split} 
\int_{\bbR^n} |x|^\g|(\Delta f)(x)|^2 \, d^n x 
& \geq - s \int_{\bbR^n} |x|^{\g-2} |(\nabla f)(x)|^2 \, d^n x     \\
& \quad + [(\g+n - 4)/4]^2 \big[(\g-n)^2+4s\big] \int_{\bbR^n} |x|^{\g-4} |f(x)|^2 \, d^n x.   \lb{3.89}
\end{split} 
\end{align}
Then, the following assertions $(i)$ and $(ii)$ hold:
\begin{align} 
\begin{split} 
& \text{$(i)$ If } \, \gamma \in \big[2-(n-1)^{1/2}, 2+(n-1)^{1/2}\big], \, \text{ inequality \eqref{3.89} holds} \\
& \quad \text{for $s \in \big[-2^{-1}\big\{(n-2)^2+(\g-2)^2\big\}, \infty\big)$,}    \lb{3.90}  
\end{split} \\
\begin{split} 
& \text{$(ii)$ If } \, \gamma \in \bbR \big\backslash \big[2-(n-1)^{1/2}, 2+(n-1)^{1/2}\big], \, \text{ inequality  \eqref{3.89} holds} \\
& \quad \text{for $s \in \big[-2^{-1}\big\{(n-2)^2-(\g-2)^2\big\}+1-n, \infty\big)$.}     \lb{3.91}
\end{split} 
\end{align} 
\end{theorem}
\begin{proof}
Choosing $\beta = 2^{-1}(\g+n-4)\big[\alpha+\g -2 - 2^{-1}(\g+n-4)\big]$ and $\tau=-1$ in \eqref{3.51}, and applying \eqref{3.41}, one obtains
\begin{align}\lb{3.92}
& \int_{\bbR^n}|x|^\g |(\Delta f)(x)|^2\, d^nx \geq \big[2^{-1}(n-\g)(\g+n-4)-\a(\a-4+2\g)\big] \no  \\
&\hspace*{3.9cm} \times \int_{\bbR^n}|x|^{\g-2}|(\nabla f)(x)|^2\, d^nx \no \\
&\quad+4^{-1}(\g+n-4)^2\big[(\a+\g-2)^2-4^{-1}(\g+n-4)^2\big] \int_{\bbR^n}|x|^{\g-4}|f(x)|^2 d^nx \\
&\quad+\big[\a^2+2(\g-2)\a+(n-1)\big]  \no  \\
&\qquad \times \int_0^\infty\int_{{\bbS^{n-1}}}r^{\g-4}\big|\big((-\Delta_{\bbS^{n-1}})^{1/2}f\big)(r,\t)\big|^2\,  d^{n-1} \omega(\theta) \, r^{n-1}dr.  \no 
\end{align}
Hence, if
\begin{equation}\lb{3.93}
0\leq \a^2+2(\g-2)\a+(n-1),
\end{equation}
one can further write
\begin{align}\lb{3.94}
\begin{split} 
& \int_{\bbR^n}|x|^\g |(\Delta f)(x)|^2\, d^nx \geq -s\int_{\bbR^n}|x|^{\g-2}|(\nabla f)(x)|^2\, d^nx \\
&\quad+16^{-1}(\g+n-4)^2\big[(\g-n)^2+4s\big] \int_{\bbR^n}|x|^{\g-4}|f(x)|^2 d^nx,   
\end{split} 
\end{align}
where we have introduced the new variable
\begin{align}\lb{3.95}
\begin{split}
s = s_{n,\g}(\alpha) &=-\big[2^{-1}(n-\g)(\g+n-4)-\a(\a-4+2\g)\big] \\
&=\alpha^2+2(\g-2)\a+ 2^{-1}\big[(\g-2)^2-(n-2)^2\big].
\end{split}
\end{align}
If the discriminant of the right-hand side of the constraint in \eqref{3.93}, $(2\g-4)^2-4(n-1)$, is nonpositive, then \eqref{3.94} holds for all $\a\in\bbR$. That is, if $\g$ is in the range given in \eqref{3.90}, then \eqref{3.89} holds for
\begin{equation}
s\in\big[-2^{-1}\big\{(n-2)^2+(\g-2)^2\big\}, \infty\big),
\end{equation}
by minimizing $s=s_{n,\g}(\a)$ with respect to $\a$, proving item $(i)$.

If, on the other hand, we now suppose $\g$ is in the range given in \eqref{3.91}, that is,
\begin{equation}
\g<2-(n-1)^{1/2}, \; \text{ or, } \; 2+(n-1)^{1/2}<\g,
\end{equation}
then the discriminant of the right-hand side of the constraint in \eqref{3.93} is nonnegative, so the constraint will hold only if $\a$ satisfies
\begin{equation}\lb{3.98}
\a\leq2-\g-\big(\g^2-4\g-n+5\big)^{1/2}, \; \text{ or, } \;
2-\g+\big(\g^2-4\g-n+5\big)^{1/2}\leq\a.
\end{equation}
Hence, \eqref{3.89} holds for all $s_{n,\g}(\a)$ such that $\a$ satisfies \eqref{3.98}, in particular, substituting the constraints \eqref{3.98} into \eqref{3.95} yields
\begin{equation}
s \in \big[-2^{-1}\big\{(n-2)^2-(\g-2)^2\big\}+1-n, \infty\big),
\end{equation}
proving item $(ii)$.
\end{proof}

\begin{remark}\lb{r3.20}
We conclude with a few remarks relating the extension of Schmincke's inequality proven in Theorem \ref{t3.19} to Corollary \ref{c3.18}.\\[1mm]
$(i)$ Letting $\g=0$, $n\geq 5$, $n\in\bbN,$ and $s=-n^2/4$ in Theorem \ref{t3.19}\,$(i)$ recovers Corollary \ref{c3.18} for $n\geq5$. \\[1mm]
$(ii)$ Letting $\g=0$, $n=4,$ and $s=-3$ in Theorem \ref{t3.19}\,$(ii)$ recovers Corollary \ref{c3.18} for $n=4$. \\[1mm]
$(iii)$ In order to recover the $n=3$ case of Corollary \ref{c3.18}, one would need to be able to choose $s=-25/36,$ however, this choice is not permitted in either case of the theorem when $\g=0$ and $n=3$. In fact, only item $(ii)$ can be applied which yields the range $s\in[-1/2,\infty)$.
\hfill $\diamond$ 
\end{remark}

Returning to Remark \ref{r3.2} once more, allowing $\tau\in(-2,0)$ (rather than a fixed $\tau=-1$ as in the proof of Theorem \ref{t3.19}) allows one to prove an improved Schmincke-type inequality. We illustrate this idea here by focusing on $\g=0$ and $n=3$ in order to recover the remaining case of $n=3$ in Corollary \ref{c3.18}.

\begin{lemma}\lb{l3.21}
Let $f\in C_0^\infty\big(\bbR^3\backslash\{0\}\big)$. Then
\begin{align}
\int_{\bbR^3} |(\Delta f)(x)|^2 \, d^3 x 
\geq - s \int_{\bbR^3} |x|^{-2} |(\nabla f)(x)|^2 \, d^3 x + 16^{-1}[4s & +(25/9)] \int_{\bbR^3} |x|^{-4} |f(x)|^2 \, d^3 x,   \no \\
&\hspace{1cm} s\in[-25/36,\infty). \lb{3.100}
\end{align}
\end{lemma}
\begin{proof}
Letting $\a\in\bbR$, $\g=0$, $n=3$, and $\tau=-(1+\varepsilon)$ with $\varepsilon\in(-1,1)$ in \eqref{3.51}, applying \eqref{3.41} once again, yields
\begin{align}\lb{3.101}
\begin{split} 
& \int_{\bbR^3} |(\Delta f)(x)|^2\, d^3 x \geq -\big(\a^2-3\a+2\b\big)\int_{\bbR^3}|x|^{-2}|(\nabla f)(x)|^2\, d^3 x  \\
&\quad+\b(2-\a-\b)\int_{\bbR^3 }|x|^{-4}|f(x)|^2\, d^3 x \\
&\quad+\big[\a^2+(\varepsilon-3)\a+2(1+\varepsilon)\beta-(3/2)(1+\varepsilon)+2-2\varepsilon^2\big]    \\
&\qquad \times \int_0^\infty\int_{\mathbb{S}^2}r^{-4}\big|\big((-\Delta_{\mathbb{S}^2})^{1/2}f\big)(r,\t)\big|^2\,  d^{2} \omega(\theta) \, r^{2}dr.  
\end{split} 
\end{align}
Next, we let $\b=A\a+B,\ A,B\in\bbR$ so that subject to the constraint
\begin{equation}\lb{3.102}
0\leq\big\{\a^2+[(2A+1)(1+\varepsilon)-4]\a+[2B-(3/2)](1+\varepsilon)+2-2\varepsilon^2\big\} ,
\end{equation}
\eqref{3.101} implies
\begin{align}\lb{3.103} 
\begin{split} 
\int_{\bbR^3} |(\Delta f)(x)|^2\, d^3 x&\geq -\big[\a^2+(2A-3)\a+2B\big]\int_{\bbR^3}|x|^{-2}|(\nabla f)(x)|^2\, d^3 x      \\
&\quad+(A\a+B)[2-B-(1+A)\a]\int_{\bbR^3}|x|^{-4}|f(x)|^2 d^3 x. 
\end{split} 
\end{align}
Next, we assume that $A,\ B,$ and $\varepsilon$ have been chosen such the discriminant of the right-hand side of \eqref{3.102} satisfies
\begin{equation}\lb{3.104}
0\leq (2A+1)^2(1+\varepsilon)^2-8(2A+1)(1+\varepsilon)-(8B-6)(1+\varepsilon)+8\varepsilon^2+8.
\end{equation}
Then the constraint \eqref{3.102} is equivalent to $\a\leq\a_-$ or $\a_+\leq\a$ where
\begin{equation}\lb{3.105}
\a_\pm=2^{-1}\big\{4-(2A+1)(1+\varepsilon)\pm\big[(2A+1)^2(1+\varepsilon)^2-2(8A+4B+1)(1+\varepsilon)+8\varepsilon^2+8\big]^{1/2}\big\}.
\end{equation}

Motivated by Corollary \ref{c3.18}, one wants to choose $A$ and $B$ in \eqref{3.103} such that
\begin{equation}
-\big[\a^2+(2A-3)\a+2B\big]=25/36,
\end{equation}
that is,
\begin{equation}\lb{3.107}
\a=\hat\a_\pm=2^{-1}\big\{3-2A\pm\big((2A-3)^2-8B-(25/9)\big)^{1/2}\big\}.
\end{equation}
Thus, one looks for $A,\ B,$ and $\varepsilon$ such that $\a_\pm=\hat\a_\pm$, that is,
\begin{align}
\begin{split}\lb{3.108}
& 4-(2A+1)(1+\varepsilon)\pm\big[(2A+1)^2(1+\varepsilon)^2-2(8A+4B+1)(1+\varepsilon)+8\varepsilon^2+8\big]^{1/2} \\
& \quad =3-2A\pm\big((2A-3)^2-8B-(25/9)\big)^{1/2}. 
\end{split}
\end{align}
Given \eqref{3.108}, one chooses $A=-1/2$ and studies the equation (after rearranging)
\begin{equation}\lb{3.109}
8\varepsilon^2-(8B-6)\varepsilon+(7/9)=0.
\end{equation}
Hence one wants
\begin{equation}\lb{3.110}
\varepsilon=\varepsilon_\pm=16^{-1}\big[8B-6\pm\big(64B^2-96B+(100/9)\big)^{1/2}\big],
\end{equation}
from which one concludes that $B$ must be chosen such that $B\leq B_-$ or $B_+\leq B$, where
\begin{equation}\lb{3.111}
B_\pm=128^{-1}\big[96\pm\big(96^2-(25600/9)\big)^{1/2}\big]=(3/4)\pm(1/3)(7/2)^{1/2}.
\end{equation}
Choosing $B=B_\pm$ in \eqref{3.110} yields
\begin{equation}\lb{3.112}
\varepsilon=\varepsilon_\pm=\pm(1/6)(7/2)^{1/2}.
\end{equation}
Finally, with $A=-1/2,$ $B=(3/4)\pm(1/3)(7/2)^{1/2}$, and $\varepsilon=\pm(1/6)(7/2)^{1/2}$, one can check that \eqref{3.104} is satisfied. By \eqref{3.107}, one has
\begin{align}
\hat\a_\pm=2\pm2^{-1}\big[(65/9)-(8/3)(7/2)^{1/2}\big]^{1/2},
\end{align}
and writing
\begin{equation}
s=s(\a)=\a^2-4\a+(3/2)+(2/3)(7/2)^{1/2},
\end{equation}
one concludes that \eqref{3.100} holds since $s(\hat\a_+)=-25/36$.
\end{proof}

Combining Theorem \ref{t3.19} and Lemma \ref{l3.21}, one finally obtains the following result.

\begin{theorem}\lb{t3.22}
Let $f\in C_0^\infty\big(\bbR^3\backslash\{0\}\big)$. Then
\begin{equation}\lb{3.115}
\int_{\bbR^3} |(\Delta f)(x)|^2 \, d^3 x 
 \geq - s \int_{\bbR^3} |x|^{-2} |(\nabla f)(x)|^2 \, d^3 x  + K(s) \int_{\bbR^3} |x|^{-4} |f(x)|^2 \, d^3 x, 
\end{equation}
where
\begin{equation}
K(s)=\begin{cases}
16^{-1}(4s+9), & s\in[-1/2,\infty),\\
16^{-1}[4s+(25/9)], & s\in[-25/36,-1/2).
\end{cases}
\end{equation}
\end{theorem}

\begin{remark} 
As expected (and by construction), choosing $s=-25/36$ in Theorem \ref{t3.22} recovers the case $n=3$ in Corollary \ref{c3.18}.
\hfill$\diamond$
\end{remark}

\appendix
\section{Weighted Hardy--Rellich inequalities and optimal constants} \lb{sA}

In this Appendix, we prove a weighted Hardy--Rellich inequality in Theorems \ref{t3.17a} and \ref{t3.18a} and show that the constants in the inequality are optimal. We begin by recalling the following simplified version of \cite[Lemma 2.1]{GMP22}: 

\begin{lemma}\lb{l3.13a}
For all $\gamma\in\bbR$ and $f\in C_0^\infty((0,\infty))$,
\begin{equation}\lb{3.49a}
\int_0^\infty r^\gamma|f''(r)|^2\, dr\geq 16^{-1}(1-\gamma)^2(3-\gamma)^2\int_0^\infty r^{\gamma-4}|f(r)|^2\, dr
\end{equation}
and
\begin{equation}\lb{3.50a}
\int_0^\infty r^\gamma|f'(r)|^2\, dr\geq 4^{-1}(1-\gamma)^2\int_0^\infty r^{\gamma-2}|f(r)|^2\, dr.
\end{equation}
Moreover, the constant on the right-hand side of \eqref{3.49a} and \eqref{3.50a} is sharp and the inequality is strict, that is, equality holds in \eqref{3.49a} and \eqref{3.50a} if and only if $f\equiv 0$.
\end{lemma}

We remark that \eqref{3.49a} and \eqref{3.50a} are the one-dimensional analogs (now on $(0,\infty)$ rather than $\bbR^n$) of \eqref{2.11} and \eqref{2.35}, respectively.

The following remark gives two general observations that will be needed in the proofs of results to follow.

\begin{remark}\lb{r3.15a}
$(i)$ Let $A,B,C,D,a\in\bbR$ satisfy $A,C,D,a>0$ and $AD-BC\geq0$. Then the function
\begin{equation}
G(t)=(At+B)/(Ct+D),\quad a\leq t,
\end{equation}
is non-decreasing on $[a,\infty)$, and hence
\begin{equation}
\inf\{G(t)\,|\, a\leq t\}=(Aa+B)/(Ca+D).
\end{equation}
$(ii)$ Let $\gamma\in\bbR$ and $n\in\bbN,\ n\geq 2$. Then
\begin{equation}
P_n(\gamma)=5\gamma^2+(2n-24)\gamma+n^2-8n+32\geq0.
\end{equation}
This follows easily from the fact that the polynomial discriminant of $P_n(\gamma)$ is $-16(n-2)^2\leq0$ for $n\geq2$.  \hfill$\diamond$
\end{remark}

\begin{lemma}\lb{l3.16a}
Let $\gamma\in\bbR$ and $f\in C_0^\infty((0,\infty)),\ n\in\bbN,\ n\geq 2$. Let $\lambda_j = j(j + n-2)$,  $j\in\bbN_0,$ once again denote the eigenvalues of $-\Delta_{{\bbS^{n-1}}}$ as in \eqref{3.3}. Then
\begin{align}
\begin{split}\lb{3.59a}
& \int_0^\infty r^{\gamma+n-1}\big|-r^{1-n}(d/dr)(r^{n-1}f'(r))+\lambda_j r^{-2}f(r)\big|^2\, dr \\
&\quad \geq\alpha_{n,\gamma,\lambda_j}\left(\int_0^\infty r^{\gamma+n-3}|f'(r)|^2\, dr+\lambda_j\int_0^\infty r^{\gamma+n-5}|f(r)|^2\, dr\right),
\end{split}
\end{align}
where
\begin{align}\lb{3.64}
\begin{split}
\alpha_{n,\gamma,\lambda_0}&=\alpha_{n,\gamma,0}=4^{-1}(n-\gamma)^2,   \\ 
\alpha_{n,\gamma,\lambda_j}&=\big[4^{-1}(n+\gamma-4)(n-\gamma)+\lambda_j\big]^2\Big/\big[4^{-1}(n+\gamma-4)^2+\lambda_j\big],\quad j\in\bbN. 
\end{split}
\end{align}
\end{lemma}
\begin{proof}
First note that
\begin{align}\lb{3.61a}
& \int_0^\infty r^{\gamma+n-1}\big|-r^{1-n}(d/dr)(r^{n-1}f'(r))+\lambda_j r^{-2}f(r)\big|^2\, dr \\
&\quad =\int_0^\infty r^{\gamma+n-1}\big[|f''(r)|^2+2(n-1)r^{-1}\Re\big(f''(r)\overline{f'(r)}\big)-2\lambda_j r^{-2}\Re\big(f''(r)\overline{f(r)}\big)   \no  \\
&\hspace{2.6cm}+(n-1)^2r^{-2}|f'(r)|^2-2\lambda_j(n-1)r^{-3}\Re\big(f'(r)\overline{f(r)}\big)+\lambda_j^2 r^{-4}|f(r)|^2\big]\, dr.  \no 
\end{align}
Furthermore, if $f=f_1+i f_2$ where $f_1$ and $f_2$ are real-valued, then
\begin{align}\lb{3.62a}
\Re\big(f''(r)\overline{f'(r)}\big)&=f_1''(r)f_1'(r)+f_2''(r)f_2'(r), \no \\
\Re\big(f''(r)\overline{f(r)}\big)&=f_1''(r)f_1(r)+f_2''(r)f_2(r),\\
\Re\big(f'(r)\overline{f(r)}\big)&=f_1'(r)f_1(r)+f_2'(r)f_2(r). \no 
\end{align}
Therefore, by \eqref{1.X} and \eqref{3.62a} we can, and will, assume without loss of generality that $f$ is real-valued for the rest of the proof. Then, by \eqref{3.61a} and \cite[Lemma 2.2]{GMP24}, 
\begin{align}\lb{3.63a}
& \int_0^\infty r^{\gamma+n-1}\big|-r^{1-n}(d/dr)(r^{n-1}f'(r)))+\lambda_j r^{-2}f(r)\big|^2\, dr  \no \\
& \quad = \int_0^\infty r^{\gamma+n-1}|f''(r)|^2\, dr+[2\lambda_j+(n-1)(1-\gamma)]\int_0^\infty r^{\gamma+n-3}|f'(r)|^2\, dr   \no  \\
&\qquad + \big[\lambda_j^2+\lambda_j(\gamma+n-4)(2-\gamma)\big]\int_0^\infty r^{\gamma+n-5}|f(r)|^2\, dr.
\end{align}
Applying Lemma \ref{l3.13a} to \eqref{3.63a} yields
\begin{align}\lb{3.64a}
\int_0^\infty & r^{\gamma+n-1}\big|-r^{1-n}(d/dr)(r^{n-1}f'(r))+\lambda_j r^{-2}f(r)\big|^2\, dr  \no \\
&\geq \big[4^{-1}(n-\gamma)^2+2\lambda_j\big]\int_0^\infty r^{\gamma+n-3}|f'(r)|^2\, dr     \\
&\quad \; +\big[\lambda_j^2+\lambda_j(\gamma+n-4)(2-\gamma)\big]\int_0^\infty r^{\gamma+n-5}|f(r)|^2\, dr.  \no 
\end{align}
Next, we want to compare \eqref{3.64a} to the integrals on the right-hand side of \eqref{3.59a}, for which we will utilize Remark \ref{r3.15a} by introducing
\begin{align}
\begin{split}
t&=t(n,\gamma,f)=\left(\int_0^\infty r^{\gamma+n-3}|f'(r)|^2\, dr\right)\left(\int_0^\infty r^{\gamma+n-5}|f(r)|^2\, dr\right)^{-1}, \\
A&=A(n,\gamma,\lambda_j)=4^{-1}(n-\gamma)^2+2\lambda_j,\\
B&=B(n,\gamma,\lambda_j)=\lambda_j^2+\lambda_j(\gamma+n-4)(2-\gamma),     \lb{3.65a} \\
C&=1,\quad D=\lambda_j,\quad j\geq 1.
\end{split}
\end{align}
Then \eqref{3.64a} and applying \eqref{3.50a} yields
\begin{align}\lb{3.66a}
\dfrac{\int_0^\infty r^{\gamma+n-1}\big|-r^{1-n}(d/dr)(r^{n-1}f'(r))+\lambda_j r^{-2}f(r)\big|^2\, dr}{\int_0^\infty r^{\gamma+n-3}|f'(r)|^2\, dr+\lambda_j\int_0^\infty r^{\gamma+n-5}|f(r)|^2\, dr} \geq (At+B)/(Ct+D),
\end{align}
where $t\geq 4^{-1}(\gamma+n-4)^2$. By Remark \ref{r3.15a} $(ii)$, one concludes that
\begin{equation}
AD-BC=\lambda_j^2+4^{-1}\lambda_j\big[5\gamma^2+(2n-24)\gamma+n^2-8n+32\big]\geq0.
\end{equation}
Hence, one can apply Remark \ref{r3.15a} $(i)$ to \eqref{3.66a}, proving \eqref{3.59a} for $\lambda_j,\ j\geq1$.

The remaining case of $\lambda_0=0$ in \eqref{3.59a} follows from directly comparing \eqref{3.64a} with the right-hand side of \eqref{3.59a}, yielding
\begin{equation}\lb{3.68a}
\alpha_{n,\gamma,\lambda_0}=\alpha_{n,\gamma,0}=4^{-1}(n-\gamma)^2,
\end{equation}
completing the proof.
\end{proof}

\begin{theorem}\lb{t3.17a}
Let $\gamma\in\bbR$ and $f\in C_0^\infty(\bbR^n\backslash \{0\}),\ n\in\bbN,\ n\geq 2$. Then
\begin{equation}\lb{3.69a}
\int_{\bbR^n} |x|^\gamma|(\Delta f)(x)|^2\, d^nx\geq A_{n,\gamma}\int_{\bbR^n}|x|^{\gamma-2}|(\nabla f)(x)|^2\, d^nx,
\end{equation}
where
\begin{equation}
A_{n,\gamma}={\min}_{j \in \bbN_0}\{\alpha_{n,\gamma,\lambda_j}\}.     \lb{A.17} 
\end{equation}
\end{theorem}
\begin{proof}
Without loss of generality, we once again assume $f$ is real-valued. We intend to utilize \eqref{3.41} in the proof, so one first notes that by \cite[Eq. (2.5)]{GMP24},
\begin{align}\lb{3.71a}
\int_{\bbR^n}|x|^{\gamma-4}|(\nabla_{\bbS^{n-1}} f)(x)|^2\, d^n x&=\int_0^\infty r^{\gamma+n-5}\int_{{\bbS^{n-1}}}|(\nabla_{\bbS^{n-1}} f(r,\dott))(\theta)|^2\, d^{n-1}\omega(\theta) \, dr   \no \\
&=\int_0^\infty r^{\gamma+n-5}(-\Delta_{\bbS^{n-1}} f(r,\dott),f(r,\dott))_{L^2({\bbS^{n-1}};d^{n-1}\omega)}\, dr  \no  \\
&=\int_0^\infty r^{\gamma+n-5}\sum_{j\in\bbN_0}\sum_{\ell=1}^{m(\lambda_j)}\lambda_j |F_{f,j,\ell}(r)|^2\, dr   \no  \\
&=\sum_{j\in\bbN_0}\sum_{\ell=1}^{m(\lambda_j)}\lambda_j \int_0^\infty r^{\gamma+n-5}|F_{f,j,\ell}(r)|^2\, dr.
\end{align}
Similarly, one has by \cite[Eq. (B.14)]{GMP24},
\begin{align}\lb{3.72a}
\int_{\bbR^n}|x|^{\gamma-2}|(\partial f/\partial r)(x)|^2\, d^n x=\sum_{j\in\bbN_0}\sum_{\ell=1}^{m(\lambda_j)} \int_0^\infty r^{\gamma+n-3}|F'_{f,j,\ell}(r)|^2\, dr.
\end{align}
On the other hand, by \cite[Eq. (2.7)]{GMP24}
\begin{align}\lb{3.73a}
\int_{\bbR^n}|x|^{\gamma}&|(\Delta f)(x)|^2\, d^n x\\
&=\sum_{j\in\bbN_0}\sum_{\ell=1}^{m(\lambda_j)} \int_0^\infty r^{\gamma+n-1}\big|-r^{1-n}(d/dr)\big(r^{n-1}F'_{f,j,\ell}(r)\big) +\lambda_j r^{-2} F_{f,j,\ell}(r)\big|^2\, dr.  \no 
\end{align}
Therefore, by \eqref{3.41}, \eqref{3.59a}, \eqref{3.71a}, \eqref{3.72a}, and \eqref{3.73a} one has
\begin{align}
& A_{n,\gamma}\int_{\bbR^n} |x|^{\gamma-2}|(\nabla f)(x)|^2\, d^nx   \no \\
& \quad = A_{n,\gamma}\left(\int_{\bbR^n}|x|^{\gamma-4}|(\nabla_{\bbS^{n-1}} f)(x)|^2+|x|^{\gamma-2}|(\partial f/\partial r)(x)|^2\, d^n x\right) \no \\
& \quad \leq \sum_{j\in\bbN_0}\sum_{\ell=1}^{m(\lambda_j)}\alpha_{n,\gamma,\lambda_j}\left(\lambda_j \int_0^\infty r^{\gamma+n-5}|F_{f,j,\ell}(r)|^2\, dr + \int_0^\infty r^{\gamma+n-3}|F'_{f,j,\ell}(r)|^2\, dr\right)  \no  \\
&\quad \leq \sum_{j\in\bbN_0}\sum_{\ell=1}^{m(\lambda_j)} \int_0^\infty r^{\gamma+n-1}\big|-r^{1-n}(d/dr)\big(r^{n-1}F'_{f,j,\ell}(r)\big) +\lambda_j r^{-2} F_{f,j,\ell}(r)\big|^2\, dr  \no  \\
& \quad = \int_{\bbR^n} |x|^\gamma|(\Delta f)(x)|^2\, d^nx.
\end{align}
\end{proof}

\begin{remark} \lb{r4.6}
Whenever $A_{n,\gamma}$ vanishes, inequality \eqref{3.69a} is rendered trivial. In this context we recall the following result from \cite[Theorem~2.3]{GPS24}: 
\begin{align} 
\begin{split}
& \int_{B_n(0;R)} |x|^\g |(\Delta f)(x)|^2\, d^n x \geq A_{n,\g} \int_{B_n(0;R)} |x|^{\g-2} |(\nabla f)(x)|^2\, d^n x  \\
& \quad + 4^{-1}\int_{B_n(0;R)} |x|^{\g-2} \Bigg(\sum_{k=1}^{N} \prod_{p=1}^{k} [\ln_{p}(\eta/|x|)]^{-2} \Bigg)|(\nabla f)(x)|^2\, d^n x    \lb{4.31} \\
& \quad + 4^{-1}\int_{B_n(0;R)}|x|^{\gamma-4}\Bigg(\sum_{k=1}^{N} \prod_{p=1}^{k} [\ln_{p}(\eta/|x|)]^{-2} \Bigg)|(\nabla_{\bbS^{n-1}} f)(x)|^2\, d^n x,   \\
& \, R \in (0,\infty), \; \gamma \in \bbR, \; N, n \in \bbN, \, n\geq2, \; \eta \in [e_{N}R,\infty), \; 
f \in C_{0}^{\infty}(B_n(0;R)\bs\{0\}).
\end{split} 
\end{align}
which yields an appropriate logarithmic refinement even if $A_{n,\gamma}$ vanishes. (For the notation employed in \eqref{4.31} see Remark \ref{r3.13}.) 
\hfill $\diamond$
\end{remark}

Of course, by restriction, the inequalities \ref{3.48a} and \eqref{4.31} extend to the case where $f \in C_0^{\infty} (B_n(0;R) \backslash \{0\})$, $n \in \bbN$, $n \geq 2$, is replaced by $f \in C_0^{\infty} (\Omega \backslash \{0\})$, where $\Omega \subseteq B_n(0;R)$ is open and bounded with $0 \in \Omega$, without changing the constants in these inequalities. 

\begin{theorem}\lb{t3.18a}
Let $R \in (0, \infty)$, $\gamma\in\bbR$, and $n\in\bbN,\ n\geq 2$. Excluding the cases $(i)$ $n=2,\ \gamma=2$ and $(ii)$ $n=3,\ \gamma=1$, the constant $A_{n,\gamma}$ in the inequality 
\begin{equation}\lb{A.23}
\int_{B_n(0;R)} |x|^\gamma|(\Delta f)(x)|^2\, d^nx\geq A_{n,\gamma}\int_{B_n(0;R)}|x|^{\gamma-2}|(\nabla f)(x)|^2\, d^nx, \quad f \in C_0^{\infty}(B_n(0;R)\backslash\{0\}), 
\end{equation}
with 
\begin{equation}
A_{n,\gamma}={\min}_{j \in \bbN_0}\{\alpha_{n,\gamma,\lambda_j}\}, 
\quad \lambda_k = k(k+n-2), \; k \in \bbN_0,     \lb{A.24} 
\end{equation}
implied by \eqref{3.69a}, \eqref{A.17}, is sharp. Here 
\begin{align}\lb{A.25}
\begin{split}
\alpha_{n,\gamma,\lambda_0}&=\alpha_{n,\gamma,0}=4^{-1}(n-\gamma)^2,   \\ 
\alpha_{n,\gamma,\lambda_j}&=\big[4^{-1}(n+\gamma-4)(n-\gamma)+\lambda_j\big]^2\Big/\big[4^{-1}(n+\gamma-4)^2+\lambda_j\big],\quad j\in\bbN. 
\end{split}
\end{align}
In particular, the constant $A_{n,\gamma}$ in Theorem \ref{t3.17a} is sharp $($i.e., inequality \eqref{A.23} remains sharp as $B_n(0;R)$ is consistently being replaced by $\bbR^n$$)$.
\end{theorem}
\begin{proof} 
We start with some notation. Let $R \in (0,\infty)$ and $\psi: (0,R) \to [0,1]$ be a $C^{\infty}$-function satisfying 
\begin{align}
& (i) \;\;\; \psi(r) = \begin{cases}
0, & 0 < r \leq R/10, \\
1, & R/5 \leq r \leq 4R/5, \\
0, & 9R/10 \leq r < R.  
\end{cases} \\
& (ii) \;\; \text{$\psi(\dott)$ is strictly increasing on $R/10 < r < R/5$.}  \\
& (iii) \; \text{$\psi(\dott)$ is strictly decreasing on $4R/5 < r < 9R/10$.}  \\
& (iv) \;\, \text{For all sufficiently small $\varepsilon > 0$, let $\psi_{\varepsilon} : (0,R) \to [0,1]$ be a $C^{\infty}$-function satisfying}   \no \\
& \hspace*{8mm} \psi_{\varepsilon} (r) = \begin{cases}
\psi(r/\varepsilon), & 0 < r \leq \varepsilon R/5, \\
1, & \varepsilon R/5 \leq r \leq 4R/5, \\
\psi(r), & 4R/5 \leq r < R.
\end{cases}
\end{align}

We start by considering the case $\gamma + n \neq 4$. Since $\lim_{j \to \infty} \alpha_{n,\gamma,\lambda_j} = \infty$ (cf.\ \eqref{A.25}), there exists $j_0 \in \bbN_0$ such that $A_{n,\gamma} = \alpha_{n,\gamma,\lambda_{j_0}}$. For all sufficiently small $\varepsilon > 0$ let $g_{\varepsilon} \in C_0^{\infty}(B_n(o;R)\backslash\{0\})$ be defined by
\begin{equation}
g_{\varepsilon}(r,\theta) = r^p \psi_{\varepsilon}(r) \varphi_{j_0,1}(\theta), 
\quad x = (r,\theta) \in B_n(0;R)\backslash\{0\},     \lb{A.30a}
\end{equation}
where 
\begin{equation}
p = p(n,\gamma,\varepsilon) = (4-n-\gamma+\varepsilon)/2.     \lb{A.31a}
\end{equation}
Then 
\begin{align}
(-\Delta g_{\varepsilon})(r,\theta) &= \bigg[- r^{1-n} \f{\partial}{\partial r}\bigg(r^{n-1} \f{\partial}{\partial r}
\big(r^p \psi_{\varepsilon}(r)\big)\bigg)\bigg] \varphi_{j_0,1}(\theta)   \no \\
& \quad + r^{p-2} \psi_{\varepsilon}(r) (- \Delta_{\bbS^{n-1}} \varphi_{j_0,1})(\theta)    \no \\
&= \bigg[-r^{1-n} \f{\partial}{\partial r}\big(p r^{p+n-2} \psi_{\varepsilon} (r) + r^{p+n-1} \psi_{\varepsilon}' (r)\big)\bigg] \varphi_{j_0,1}(\theta)    \no \\
& \quad + \lambda_{j_0} r^{p-2} \psi_{\varepsilon} (r) \varphi_{j_0,1}(\theta)     \no \\
&= \big\{- r^{1-n} \big[p(p+n-2) r^{p+n-3} \psi_{\varepsilon}(r) + (2p+n-1)r^{p+n-2} \psi_{\varepsilon}' (r)   \no \\
& \hspace*{7mm} + r^{p+n-1} \psi_{\varepsilon}'' (r)\big]\big\} \varphi_{j_0,1}(\theta) 
+ \lambda_{j_0} r^{p-2} \psi_{\varepsilon}(r) \varphi_{j_0,1}(\theta)    \no \\
&= [\lambda_{j_0} - p(p+n-2)] r^{p-2} \psi_{\varepsilon}(r) \varphi_{j_0,1}(\theta) - 
(2p+n-1) r^{p-1} \psi_{\varepsilon}'(r) \varphi_{j_0,1}(\theta)    \no \\
& \quad - r^p \psi_{\varepsilon}''(r) \varphi_{j_0,1}(\theta),     \lb{A.30} 
\end{align}
and by \eqref{3.41},
\begin{align}
|(\nabla g_{\varepsilon})(r,\theta)|^2 &= \bigg|\f{\partial}{\partial r} \big[r^p \psi_{\varepsilon}(r) \big]\bigg|^2 
|\varphi_{j_0,1}(\theta)|^2 + r^{2p-2} \psi_{\varepsilon}(r)^2 |(\nabla_{\bbS^{n-1}} \varphi_{j_0,1})(\theta)|^2    \no \\
&= \big\{p^2 r^{2p-2} \psi_{\varepsilon}(r)^2 + 2p r^{2p-1} \psi_{\varepsilon}(r) \psi_{\varepsilon}'(r)
+ r^{2p} [\psi_{\varepsilon}'(r)]^2\big\} |\varphi_{j_0,1}(\theta)|^2    \no \\
&\quad + r^{2p-2} \psi_{\varepsilon}(r)^2 |(\nabla_{\bbS^{n-1}} \varphi_{j_0,1})(\theta)|^2.    \lb{A.31}
\end{align}
Thus, by \eqref{A.30},
\begin{align}
& \int_{B_n(0;R)} |x|^{\gamma} |(\Delta g_{\varepsilon})(x)|^2 \, d^n x     \no \\
& \quad = \int_0^R \int_{\bbS^{n-1}} r^{\gamma+n-1} \big\{[\lambda_{j_0} - p(p+n-2)] r^{p-2} \psi_{\varepsilon}(r) - (2p+n-1) r^{p-1} \psi_{\varepsilon}'(r)   \no \\
& \hspace*{3.55cm} - r^p \psi_{\varepsilon}''(r)\big\}^2  |\varphi_{j_0,1}(\theta)|^2 d^{n-1} \omega(\theta) \, dr 
\no \\
& \quad = \int_0^R r^{\gamma+n-1} \big\{[\lambda_{j_0} - p(p+n-2)]^2 r^{2p-4} \psi_{\varepsilon}(r)^2   \no \\
& \hspace*{2.7cm} - 2 [\lambda_{j_0} - p(p+n-2)] (2p+n-1) r^{2p-3} \psi_{\varepsilon}(r)\psi_{\varepsilon}'(r)  \no \\
& \hspace*{2.7cm} - 2 [\lambda_{j_0} - p(p+n-2)] r^{2p-2} \psi_{\varepsilon}(r)\psi_{\varepsilon}''(r) \no \\
& \hspace*{2.7cm} + (2p+n-1)^2 r^{2p-2} [\psi_{\varepsilon}'(r)]^2 + 2 (2p+n-1) r^{2p-1} \psi_{\varepsilon}'(r)\psi_{\varepsilon}''(r)   \no \\
& \hspace*{2.7cm} + r^{2p} [\psi_{\varepsilon}''(r)]^2\big\} \,dr    \no \\
& \quad = \big[\lambda_{j_0} - 4^{-1} (4-n-\gamma+\varepsilon)(n-\gamma+\varepsilon)\big]^2 
\int_0^R r^{-1 + \varepsilon} \psi_{\varepsilon}(r)^2 \, dr   \no \\
& \qquad - 2 \big[\lambda_{j_0} - 4^{-1} (4-n-\gamma+\varepsilon)(n-\gamma+\varepsilon)\big] (3-\gamma+\varepsilon) \int_0^R r^{\varepsilon} \psi_{\varepsilon}(r) \psi_{\varepsilon}'(r) \, dr   \no \\
& \qquad - 2 \big[\lambda_{j_0} - 4^{-1} (4-n-\gamma+\varepsilon)(n-\gamma+\varepsilon)\big] 
\int_0^R r^{1+\varepsilon} \psi_{\varepsilon}(r) \psi_{\varepsilon}''(r) \, dr   \no \\
& \qquad + (3-\gamma+\varepsilon)^2 \int_0^R r^{1+\varepsilon} [\psi_{\varepsilon}'(r)]^2 \, dr   \no \\
& \qquad + 2(3-\gamma+\varepsilon) \int_0^R r^{2+\varepsilon} \psi_{\varepsilon}'(r) \psi_{\varepsilon}''(r) \, dr    \no \\ 
& \qquad + \int_0^R r^{3+\varepsilon} [\psi_{\varepsilon}''(r)]^2 \, dr,     \lb{A.32}
\end{align}
and by \eqref{A.31},
\begin{align}
& \int_{B_n(0;R)} |x|^{\gamma-2} |(\nabla g_{\varepsilon})(x)|^2 \, d^nx   \no \\
& \quad = \int_0^R \int_{\bbS^{n-1}} r^{\gamma+n-3} \big\{\big[p^2 r^{2p-2} \psi_{\varepsilon}(r)^2 
+ 2p r^{2p-1} \psi_{\varepsilon}(r) \psi_{\varepsilon}'(r) + r^{2p} [\psi_{\varepsilon}'(r)]^2 \big] \big\} 
|\varphi_{j_0,1}(\theta)|^2     \no \\
& \hspace*{3.6cm} + r^{2p-2} \psi_{\varepsilon}(r)^2 |(\nabla_{\bbS^{n-1}} \varphi_{j_0,1})(\theta)|^2\big\} 
\, d^{n-1} \omega(\theta) \, dr    \no \\
& \quad = \int_0^R \big\{p^2 r^{\gamma+n+2p-5} \psi_{\varepsilon}(r)^2 
+ 2p r^{\gamma + n + 2p - 4} \psi_{\varepsilon}(r)\psi_{\varepsilon}'(r) 
+ r^{\gamma + n + 2p -3}[\psi_{\varepsilon}'(r)]^2    \no \\
& \hspace*{1.55cm} + \lambda_{j_0} r^{\gamma+n+2p-5} \psi_{\varepsilon}(r)^2  \big\} \, dr    \no \\
& \quad = \big[4^{-1} (4-n-\gamma+\varepsilon)^2 + \lambda_{j_0}\big] \int_0^R r^{-1+\varepsilon} 
\psi_{\varepsilon}(r)^2 \, dr + (4-n-\gamma+\varepsilon) \int_0^R r^{\varepsilon} \psi_{\varepsilon}(r)\psi_{\varepsilon}'(r) \, dr   \no \\ 
& \qquad + \int_0^R r^{1+\varepsilon} [\psi_{\varepsilon}'(r)]^2 \, dr.     \lb{A.33} 
\end{align} 
It is clear that one can choose a decreasing sequence $\{\varepsilon_k\}_{k \in \bbN}$, with $\lim_{k\to\infty}\varepsilon_k = 0$ such that 
\begin{equation}
\varphi_{\varepsilon_{k+1}}(r) \geq \varphi_{\varepsilon_k}(r), \quad r \in (0,R), \; k \in \bbN. 
\end{equation}
For $\ell \in \bbN$, $1 \leq \ell \leq 21$, we denote constants $c_{\ell}$ that depend on $\psi, n, \gamma, \lambda_{j_0},R$, but are independent of $\varepsilon_k$, $k \in \bbN$. Then one obtains
\begin{align}
& \lim_{k\to\infty} \int_0^R r^{-1+\varepsilon_k} \psi_{\varepsilon_k}(r)^2 \, dr = \infty,   \lb{A.35} \\
& \bigg|\int_0^R r^{\varepsilon_k} \psi_{\varepsilon_k}(r)\psi_{\varepsilon_k}'(r) \, dr\bigg| \leq 
\int_{\varepsilon_kR/10}^{\varepsilon_kR/5} r^{\varepsilon_k} \varepsilon_k^{-1} \, dr \, c_1 +
\int_{4R/5}^{9R/10} r^{\varepsilon_k} |\psi'(r)| \, dr    \no \\
& \hspace*{3.7cm} \leq c_1 \varepsilon_k^{-1} (1+\varepsilon_k)^{-1}\big[(R/5)^{1+\varepsilon_k} 
- (R/10)^{1+\varepsilon_k}\big] \varepsilon_k^{1+\varepsilon_k} + c_2   \no \\
& \hspace*{3.7cm} \leq c_3 \varepsilon_k^{\varepsilon_k} (1+\varepsilon_k)^{-1} + c_2 \leq c_4,    \lb{A.36} \\
& \bigg|\int_0^R r^{1+\varepsilon_k} \psi_{\varepsilon_k}(r)\psi_{\varepsilon_k}''(r) \, dr\bigg| \leq 
\int_{\varepsilon_kR/10}^{\varepsilon_kR/5} r^{1+\varepsilon_k} \varepsilon_k^{-2} \, dr \, c_5 +
\int_{4R/5}^{9R/10} r^{1+\varepsilon_k} |\psi''(r)| \, dr    \no \\
& \hspace*{4.05cm} \leq c_5 \varepsilon_k^{-2} (2+\varepsilon_k)^{-1}\big[(R/5)^{2+\varepsilon_k} 
- (R/10)^{2+\varepsilon_k}\big] \varepsilon_k^{2+\varepsilon_k} + c_6   \no \\
& \hspace*{4.05cm} \leq c_7 \varepsilon_k^{\varepsilon_k} (2+\varepsilon_k)^{-1} + c_6 \leq c_8,     \lb{A.37} \\
& \bigg|\int_0^R r^{1+\varepsilon_k} [\psi_{\varepsilon_k}'(r)]^2 \, dr\bigg| \leq 
\int_{\varepsilon_kR/10}^{\varepsilon_kR/5} r^{1+\varepsilon_k} \varepsilon_k^{-2} \, dr \,c_9 +
\int_{4R/5}^{9R/10} r^{1+\varepsilon_k} [\psi'(r)]^2 \, dr    \no \\
& \hspace*{3.4cm} \leq c_9 \varepsilon_k^{\varepsilon_k} (2+\varepsilon_k)^{-1} + c_{10} \leq c_{11},    
\lb{A.38} \\
& \bigg|\int_0^R r^{2+\varepsilon_k} \psi_{\varepsilon_k}'(r)\psi_{\varepsilon_k}''(r) \, dr\bigg| \leq 
\int_{\varepsilon_kR/10}^{\varepsilon_kR/5} r^{2+\varepsilon_k} \varepsilon_k^{-3} \, dr \, c_{12} +
\int_{4R/5}^{9R/10} r^{2+\varepsilon_k} |\psi'(r) \psi''(r)| \, dr    \no \\
& \hspace*{4.05cm} \leq c_{12} \varepsilon_k^{-3} (3+\varepsilon_k)^{-1}\big[(R/5)^{3+\varepsilon_k} 
- (R/10)^{3+\varepsilon_k}\big] \varepsilon_k^{3+\varepsilon_k} + c_{13}   \no \\
& \hspace*{4.05cm} \leq c_{14} \varepsilon_k^{\varepsilon_k} (3+\varepsilon_k)^{-1} + c_{13} \leq c_{15},   
\lb{A.39} \\
& \bigg|\int_0^R r^{3+\varepsilon_k} [\psi_{\varepsilon_k}''(r)]^2 \, dr\bigg| \leq 
\int_{\varepsilon_kR/10}^{\varepsilon_kR/5} r^{3+\varepsilon_k} \varepsilon_k^{-4} \, dr \, c_{16} +
\int_{4R/5}^{9R/10} r^{3+\varepsilon_k} [\psi''(r)]^2 \, dr    \no \\
& \hspace*{3.45cm} \leq c_{16} \varepsilon_k^{-4} (4+\varepsilon_k)^{-1}\big[(R/5)^{4+\varepsilon_k} 
- (R/10)^{4+\varepsilon_k}\big] \varepsilon_k^{4+\varepsilon_k} + c_{17}   \no \\
& \hspace*{3.45cm} \leq c_{18} \varepsilon_k^{\varepsilon_k} (4+\varepsilon_k)^{-1} + c_{17} \leq c_{19}.   
\lb{A.40}
\end{align}
Thus, \eqref{A.32}, \eqref{A.33}, \eqref{A.35}--\eqref{A.40} imply 
\begin{align}
& \int_{B_n(0;R)} |x|^{\gamma} |(\Delta g_{\varepsilon_k})(x)|^2 \, d^n x     \lb{A.41} \\
& \quad =  \big[\lambda_{j_0} - 4^{-1} (4-n-\gamma+\varepsilon_k)(n-\gamma+\varepsilon_k)\big]^2 
\int_0^R r^{-1 + \varepsilon_k} \psi_{\varepsilon_k}(r)^2 \, dr 
+ F(\lambda_{j_0},n,\gamma,R,\psi,\varepsilon_k), 
\no
\end{align}
and 
\begin{align}
& \int_{B_n(0;R)} |x|^{\gamma-2} |(\nabla g_{\varepsilon_k})(x)|^2 \, d^n x     \no \\
& \quad =  \big[4^{-1} (4-n-\gamma+\varepsilon_k)^2 + \lambda_{j_0}\big] 
\int_0^R r^{-1 + \varepsilon_k} \psi_{\varepsilon_k}(r)^2 \, dr + G(\lambda_{j_0},n,\gamma,R,\psi,\varepsilon_k),
\lb{A.42} 
\end{align}
where 
\begin{equation}
|F(\lambda_{j_0},n,\gamma,R,\psi,\varepsilon_k)| \leq c_{20}, \quad 
|G(\lambda_{j_0},n,\gamma,R,\psi,\varepsilon_k)| \leq c_{21}.    \lb{A.43} 
\end{equation}
From \eqref{A.35} and \eqref{A.41}--\eqref{A.43} one then obtains 
\begin{align}
\lim_{k\to\infty}\f{\int_{B_n(0;R)} |x|^{\gamma} |(\Delta g_{\varepsilon_k})(x)|^2 \, d^n x}{\int_{B_n(0;R)} |x|^{\gamma-2} |(\nabla g_{\varepsilon_k})(x)|^2 \, d^n x} = \f{\big[\lambda_{j_0} + 4^{-1} (n+\gamma-4)(n-\gamma)\big]^2}{\lambda_{j_0} + 4^{-1}(n+\gamma-4)^2} = \alpha_{n,\gamma,\lambda_{j_0}} 
= A_{n,\gamma},     \lb{A.44} 
\end{align}
hence, $A_{n,\gamma}$ is sharp in the case $\gamma+n \neq 4$.  

It remains to treat the case $\gamma+n = 4$. Then 
\begin{equation}
A_{n,4-n} = \min\big\{(n-2)^2,\lambda_1\big\} = \min\big\{(n-2)^2,n-1\big\}. 
\end{equation}
If $A_{n,4-n} = (n-2)^2$, then $(n-2)^2 \leq n-1$, hence either $n=2$ and $\gamma = 2$, or, $n=3$ and $\gamma=1$, which, however, are the two excluded cases. Thus, we assume that 
\begin{equation}
A_{n,4-n} = n-1 = \lambda_1 > 0. 
\end{equation}
But in this case the arguments in the proof for the case $\gamma+n \neq 4$ still apply and one again arrives at \eqref{A.44}, proving that $A_{n,4-n}$ is sharp. 

Since $C_0^{\infty}(B_n\backslash\{0\}) \subset C_0^{\infty}(\bbR^n\backslash\{0\})$, the constant $A_{n,\gamma}$ in Theorem \ref{t3.17a} is sharp. 
\end{proof}

\begin{remark}\lb{r3.20a}
Using arguments analogous to those in the proof of Theorem \ref{t3.18a}, one can show that the constant 
$C_{n,\gamma}$ in the inequality 
\begin{equation}\lb{A.47}
\int_{B_n(0;R)} |x|^\g|(\Delta f)(x)|^2 \, d^n x 
 \geq C_{n,\g} \int_{B_n(0;R)} |x|^{\g-4} |f(x)|^2 \, d^n x, \quad f \in C_0^{\infty}(B_n(0;R)\backslash\{0\}), 
\end{equation}
where
\begin{equation}\lb{A.50}
C_{n,\g}=\min_{j\in\bbN_0}\Big\{\big[4^{-1}(n-2)^2-4^{-1}(\g-2)^2+j(j+n-2)\big]^2\Big\},
\end{equation}
implied by \eqref{3.44}, \eqref{3.45}, is sharp.

To see this one notes that 
\begin{equation}
\lim_{j \to \infty} \big[\lambda_j + 4^{-1} (n-2)^2 - 4^{-1}(\gamma-2)^2\big]^2 = \infty, \quad 
\lambda_k = k(k+n-2), \; k \in \bbN_0, 
\end{equation}
and hence there exists $j_0 \in \bbN_0$ such that 
\begin{equation}
C_{n,\g}= \big[4^{-1}(n-2)^2 - 4^{-1}(\gamma-2)^2 + \lambda_{j_0}\big]^2.    \lb{A.52}
\end{equation}
By \eqref{A.30a}, \eqref{A.31a} one obtains 
\begin{align}
\begin{split} 
\int_{B_n(0;R)} |x|^{\gamma-4} |g_{\varepsilon_k}(x)|^2 \, d^n x&= \int_0^{\infty} \int_{\bbS^{n-1}} r^{\gamma+n-5+2p} \psi_{\varepsilon_k}(r)^2 |\varphi_{j_0,1}(\theta)|^2 \, d^{n-1} \omega(\theta) \, dr  
\lb{A.53} \\
&= \int_0^R r^{-1+\varepsilon_k} \psi_{\varepsilon_k}(r)^2 \, dr \underset{k \to \infty}{\longrightarrow} \infty. 
\end{split}
\end{align}
By \eqref{A.35}, \eqref{A.41}, \eqref{A.43}, \eqref{A.52}, and \eqref{A.53} one infers 
\begin{equation}
\lim_{k\to\infty} \f{\int_{B_n(0;R)} |x|^\g|(\Delta g_{\varepsilon_k})(x)|^2 \, d^n x }{\int_{B_n(0;R)} |x|^{\gamma-4} |g_{\varepsilon_k}(x)|^2 \, d^n x} = \big[4^{-1} (n+\gamma-4)(n-\gamma)+\lambda_{j_0}\big]^2 
= C_{n,\gamma},
\end{equation}
and hence $C_{n,\gamma}$ is sharp. 

In particular, the constant $C_{n,\gamma}$ in  \eqref{3.44} remains sharp as $B_n(0;R)$ is consistently being replaced by $\bbR^n$, recovering once more a result in \cite[p.~148--149, Theorem~3.1]{CM12}. 
\hfill$\diamond$
\end{remark}

We now give two lemmas regarding the number of terms one needs to check when calculating $A_{n,\gamma}$ in Theorem \ref{t3.17a}.

\begin{lemma}\lb{l4.9}
Let $\gamma\in\bbR,\ n\geq2,\ n\in\bbN$. Then in Theorem \ref{t3.17a},
\begin{align}\lb{4.55}
A_{n,\gamma}&={\min}_{j \in \bbN_0}\{\alpha_{n,\gamma,\lambda_j}\}  \no \\
&=\min\{\alpha_{n,\gamma,\lambda_0},\alpha_{n,\gamma,\lambda_1}\}  \\
&=\min\{\alpha_{n,\gamma,0},\alpha_{n,\gamma,n-1}\}  \no ,
\end{align}
if for all $j \in \mathbb{N}$, one of the following three conditions is satisfied$:$\\[1mm]
$(i)$ $\gamma\in[-8j+2,2];$\\[1mm]
$(ii)$ $\gamma\in\bbR\backslash[-8j+2,2]$ and
\begin{equation}
n\geq -2(\g-3+j)+2\big[4^{-1}\g^2+(2j-1)\gamma-4j+1\big]^{1/2};
\end{equation}
$(iii)$ $\gamma\in\bbR\backslash[-8j+2,2]$ and
\begin{equation}
n\leq -2(\g-3+j)-2\big[4^{-1}\g^2+(2j-1)\gamma-4j+1\big]^{1/2}.
\end{equation}
\end{lemma}
\begin{proof}
From \eqref{3.64}, we define for $\gamma\in\bbR,\ n\geq2,\ n\in\bbN$,
\begin{equation}
f_{n,\g}(j)=\big[4^{-1}(n+\gamma-4)(n-\gamma)+\lambda_j\big]^2/\big[4^{-1}(n+\gamma-4)^2+\lambda_j\big],\quad j\in\bbN
\end{equation}
We want to study when $f_{n,\g}(j)$ is an increasing function in $j$, so we consider when $f'_{n,\g}(j)\geq0$, which, after some straightforward calculations, is equivalent to
\begin{equation}\lb{4.59}
(1/4)n^2+(\g-3+j)n+8+j(j-2)+\g[(3/4)\g-5]\geq0.
\end{equation}
Now one considers the discriminant
\begin{align}
\begin{split}
g_j(\gamma)&=(\g-3+j)^2-4(1/4)\{8+j(j-2)+\g[(3/4)\g-5]\}\\
&=(1/4)\g^2+(2j-1)\g-4j+1,\quad \gamma\in\bbR,\ j\in\bbN.
\end{split}
\end{align}
To understand the sign of $g_j(\gamma)$, one next considers the discriminant
\begin{equation}
h(j)=(2j-1)^2-4(1/4)(1-4j)=4j^2,\quad j\in\bbN.
\end{equation}
Thus $g_j(\g)\leq 0$ for all $\g\in\bbR$ such that
\begin{equation}
\gamma\leq-2(2j-1)+2\sqrt{4j^2}=2,
\end{equation}
or, 
\begin{equation}
\gamma\geq-2(2j-1)-2\sqrt{4j^2}=-8j+2.
\end{equation}
Thus, by \eqref{4.59},  $f'_{n, \gamma}(j) \geq 0$ if Condition $(i)$ holds.

On the other hand, if $\g\in\bbR\backslash[-8j+2,2],\ j\in\bbN$, then by \eqref{4.59}, $f'_{n, \gamma}(j) \geq 0$ if
\begin{equation}
n\geq -2(\g-3+j)+2\big[4^{-1}\g^2+(2j-1)\gamma-4j+1\big]^{1/2},
\end{equation}
or, 
\begin{equation}
n\leq -2(\g-3+j)-2\big[4^{-1}\g^2+(2j-1)\gamma-4j+1\big]^{1/2},
\end{equation}
which are Conditions $(ii)$ and $(iii)$, respectively.
\end{proof}

Considering the proof of Lemma \ref{l4.9}, one notices that Lemma \ref{l4.9} can be extended as follows. If $(i)$--$(iii)$ in Lemma \ref{l4.9} hold for all $j\geq j_0$ for some $j_0\in\bbN$, then only the first $j_0+1$ terms need to be checked rather than all $j\in\bbN$. This is the content of the following result.

\begin{lemma}\lb{l4.10}
Let $\gamma\in\bbR,\ n\geq2,\ n\in\bbN$.
Suppose there exists $j_0 \in \mathbb{N}$ such that for all $j \in \mathbb{N}$ with $j \geq j_0$, one of the following three conditions is satisfied$:$\\[1mm]
$(i)$ $\gamma\in[-8j+2,2];$\\[1mm]
$(ii)$ $\gamma\in\bbR\backslash[-8j+2,2]$ and
\begin{equation}
n\geq -2(\g-3+j)+2\big[4^{-1}\g^2+(2j-1)\gamma-4j+1\big]^{1/2};
\end{equation}
$(iii)$ $\gamma\in\bbR\backslash[-8j+2,2]$ and
\begin{equation}
n\leq -2(\g-3+j)-2\big[4^{-1}\g^2+(2j-1)\gamma-4j+1\big]^{1/2}.
\end{equation}
Then in Theorem \ref{t3.17a},
\begin{align}\lb{4.66}
A_{n,\gamma}&={\min}_{j \in \bbN_0}\{\alpha_{n,\gamma,\lambda_j}\}  \no \\
&={\min}_{j \in \bbN_0, \, 0\leq j\leq j_0}\{\alpha_{n,\gamma,\lambda_j}\}.
\end{align}
\end{lemma}

The following classical result now follows as a corollary of Theorem \ref{t3.17a} and an application of Lemma \ref{l4.9}: 

\begin{corollary}\lb{c3.18}
Let $f\in C_0^\infty(\bbR^n\backslash\{0\})$, $n\in\bbN$, $n\geq2$. Then
\begin{equation}\lb{3.86}
\int_{\bbR^n} |(\Delta f)(x)|^2\, d^nx\geq A_{n,0}\int_{\bbR^n}|x|^{-2}|(\nabla f)(x)|^2\, d^nx,
\end{equation}
where
\begin{equation}\lb{3.87}
A_{n,0}=\begin{cases}
n^2/4, & n\geq 5,\\
3,& n=4,\\
25/36,& n=3, \\
0, & n=2, 
\end{cases}
\end{equation}
and the constant $A_{n,0}$ is sharp. In particular, \eqref{3.86} reduces to a trivial inequality for $n=2$, that is, there is no  $A_{2,0} \in (0,\infty)$ such that \eqref{3.86} holds for $n=2$.
\end{corollary}
\begin{proof}
The result follows from letting $\g=0$ in Theorem \ref{t3.17a} (with sharpness as an immediate consequence) and applying Lemma \ref{l4.9}. In particular, since $\g=0$ satisfies Condition $(i)$ of Lemma \ref{l4.9} for all $j\in\bbN$, from \eqref{4.55} one has
\begin{equation}
A_{n,0}=\min\{\a_{n,0,0},\a_{n,0,n-1}\},
\end{equation}
where
\begin{align}\lb{3.107a}
\begin{split}
\alpha_{n,0,n-1}&=\big[4^{-1}n(n-4)+n-1\big]^2/\big[4^{-1}(n-4)^2+n-1\big],\\
\alpha_{n,0,\lambda_0}&=\alpha_{n,0,0}=n^2/4.
\end{split}
\end{align}
Hence, setting $n=3,4$ in \eqref{3.107a} shows the minimums occur when $j=1$ yielding $A_{3,0}=25/36$ and $A_{4,0}=3$. Furthermore, setting $n=2$ in \eqref{3.107a} yields a minimum of zero at $j=1$ showing that \eqref{3.86} does not hold for any $A_{2,0}>0$ (see also \cite{Ca20} in this context).

Finally, we note that for $n\geq 5$,
\begin{equation}
\alpha_{n,0,\lambda_1}=\frac{\big[4^{-1}n(n-4)+n-1\big]^2}{4^{-1}(n-4)^2+n-1}=\frac{n^4-8n^2+16}{4n^2-16n+48}=\frac{n^2}{4}\cdot\frac{n^2-8+16n^{-2}}{n^2-4n+12}\geq\frac{n^2}{4}=\alpha_{n,0,0},
\end{equation}
since $16n^{-2}\geq -4n+20$ for $n\geq 5$, implying 
\begin{equation}
n^2-8+16n^{-2}\geq n^2-4n+12,\quad n\geq5.
\end{equation}
Therefore, for $n\geq 5$, the minimum occurs at $j=0$ implying $A_{n,0}=n^2/4,\ n\geq5$.
\end{proof}

\begin{remark} \lb{r4.13}
$(i)$ The inequalities proven in this paper were typically formulated for the smallest natural function space such as $f \in C_0^{\infty}(\bbR^n \backslash \{0\})$, or, $f \in C_{0}^{\infty}(B_n(0;R)\bs\{0\})$. Thus, at least  in principle, the optimal constants could have increased in the process when compared to the function spaces $f \in C_0^{\infty}(\bbR^n)$, or, $f \in C_{0}^{\infty}(B_n(0;R))$ typically employed in \cite{AGS06}, \cite{Be08a}, \cite{Ca20}, \cite{GM11}, \cite[]{GM13}, \cite{TZ07}, etc. Interestingly enough, the result in \cite{CM12} and those in the present paper demonstrate this possible increase in optimal constants is not happening. \\[1mm] 
$(ii)$ We conclude by mentioning that factorization would also work for other homogeneous and singular interactions, for instance, for point dipole interactions, where $|x|^{-2}$ is replaced by $|x|^{-3}(c \cdot x)$, with $c \in \bbR^n$ a constant vector (cf.\ \cite{AG21}). 
\hfill $\diamond$
\end{remark}

\medskip



\end{document}